\documentclass[12pt]{amsart}
\usepackage{fullpage}
\usepackage{amssymb,amsmath,amsthm,bbold,mathtools,bm,extpfeil}
\usepackage{enumitem}
\usepackage{comment}
\usepackage{tikz} 
\usetikzlibrary{cd}
\usetikzlibrary{arrows}

\newtheorem{defn}{Definition}[section]

\newtheorem{rmk}[defn]{Remark}
\newtheorem{lem}[defn]{Lemma}
\newtheorem{prop}[defn]{Proposition}
\newtheorem{thm}[defn]{Theorem}

\numberwithin{equation}{section}

\theoremstyle{plain}

\newcommand{\bfd}{\mathbf{d}}
\newcommand{\bfi}{\mathbf{i}}
\newcommand{\bfh}{\mathbf{h}}
\newcommand{\bfp}{\mathbf{p}}
\newcommand{\bfq}{\mathbf{q}}
\newcommand{\bfr}{\mathbf{r}}
\newcommand{\bfs}{\mathbf{s}}

\newcommand{\bfdelta}{\bm{\delta}}

\newcommand{\bfC}{\mathbf{C}}
\newcommand{\bfF}{\mathbf{F}}

\newcommand{\bfW}{\mathbf{W}}

\newcommand{\fg}{\mathfrak{g}}
\newcommand{\fh}{\mathfrak{h}}

\newcommand{\fk}{\mathfrak{k}}

\newcommand{\fm}{\mathfrak{m}}
\newcommand{\fn}{\mathfrak{n}}

\newcommand{\sE}{\mathcal{E}}
\newcommand{\sF}{\mathcal{F}}

\newcommand{\sI}{\mathcal{I}}

\newcommand{\sM}{\mathcal{M}}

\newcommand{\sO}{\mathcal{O}}
\newcommand{\sR}{\mathcal{R}}
\newcommand{\sT}{\mathcal{T}}

\newcommand{\R}{\mathbb{R}}

\newcommand{\pr}{\mathrm{pr}}

\newcommand{\Sym}{\mathrm{Sym}}
\newcommand{\ev}{\mathrm{ev}}
\newcommand{\Hom}{\mathrm{Hom}}
\newcommand{\Gr}{\mathrm{Gr}}

\newcommand{\CE}{\mathrm{CE}}
\newcommand{\id}{\mathrm{id}}

\newcommand{\sh}{\mathrm{S}}
\newcommand{\psh}{\mathrm{P}}

\newcommand{\wSym}{\widehat{\Sym}}

\newcommand{\wotimes}{\widehat{\otimes}}
\newcommand{\wOmega}{\widehat{\Omega}}

\newcommand{\uR}{{R^\sharp}}
\newcommand{\uS}{{S^\sharp}}
\newcommand{\uC}{{C^\sharp}}

\usepackage{hyperref}
\usepackage[alphabetic,initials]{amsrefs}

\title{Homotopy theory for curved $L_\infty$ spaces}
\author{Shuhan Jiang}
\address{
	Department of Mathematics, University of Zurich, Winterthurerstrasse 190, CH-8057 Zurich, Switzerland
}
\email{shuhan.jiang@math.uzh.ch}

\begin{document}
	
	\begin{abstract}
		This paper proves that $L_\infty$ spaces over a dg manifold form a category of fibrant objects. Together with the first main result of the companion paper \cite{cattaneojiang26}, this implies that transitive $L_\infty$ algebroids over a dg manifold also form a category of fibrant objects.
	\end{abstract}
	
	\maketitle
	\setcounter{tocdepth}{1}
	\tableofcontents
	
	%Need to modify everything to be defined over a cdga.
	
	\section{Introduction}
	
	In \cite{cattaneojiang26}, we establish an equivalence between the fibered category of transitive $L_\infty$ algebroids over dg manifolds and the split fibered category of $L_\infty$ spaces over dg manifolds
	\[
	\mathbf{L_\infty Algd}_{\mathrm{fib}} \cong \mathbf{L_\infty Sp},
	\]
	which detects weak equivalences. We further construct a faithful functor
	\[
	\mathbf{Fib}\colon \mathbf{L_\infty Algd} \longrightarrow \mathbf{L_\infty Algd}_{\mathrm{fib}} \cong \mathbf{L_\infty Sp},
	\]
	which also detects weak equivalences. 
	
	In the present work, we study the homotopy theory for $L_\infty$ spaces over a dg manifold. 
	Our main result shows that $\mathbf{Fib}$ can be interpreted as a fibrant replacement functor.
	\begin{thm}
		$L_\infty$ spaces over a dg manifold form a category of fibrant objects (CFO). 
	\end{thm}
	
	The proof proceeds by using the global sections functor to translate the problem into an algebraic one, and then applying techniques from the theory of $L_\infty$ algebras over a field \cites{getzler2009, rogers2020explicit,rogers2023complete,getzler2025higherholonomycurvedlinftyalgebras}, adapted to the setting of filtered commutative differential graded algebras (filtered cdgas). In particular, we provide in the appendices a version of the homotopy transfer theorem for curved $L_\infty$ algebras over filtered cdgas.
	%, needed for the construction of path space objects.
	
	It is known that dg manifolds also form a CFO \cites{Behrend2020thx,carchedi2023derivedmanifoldsdifferentialgraded}. It is natural to ask whether the base and fiber CFO structures of $\mathbf{L_\infty Sp}$ assemble into a total CFO structure. Further evidence for this expectation comes from the fact that $L_\infty$ groupoids in the Banach manifold setting form an incomplete CFO \cite{rogers2020homotopy}. 
	%The incompleteness arises from the lack of pullbacks in the category of manifolds, a deficiency that is remedied in the setting of dg manifolds.
	
	%dg manifolds has a site structure inherited from the category of manifolds.
	
 	\subsection*{Acknowledgments}
	
	The author is grateful to Ezra Getzler for a helpful discussion.
	
	The author acknowledges partial support of the SNF Grant No. 200021 227719 and of the Simons Collaboration on Global Categorical Symmetries. This research was (partly) supported by the NCCR SwissMAP, funded by the Swiss National Science Foundation. This article is based upon work from COST Action 21109 CaLISTA, supported by COST (European Cooperation in Science and Technology) (www.cost.eu), MSCA-2021-SE-01-101086123 CaLIGOLA, and MSCA-DN CaLiForNIA-101119552.

	\section{Curved \texorpdfstring{$L_\infty$}{L-infinity}  spaces}
	
	Throughout this paper, all filtrations are assumed to be descending, separated, exhaustive, and complete. 
	
	\subsection{Curved \texorpdfstring{$L_\infty$}{L-infinity}  algebras}
	
	Let $R=(\uR, D_R)$ be a filtered commutative differential graded algebra (filtered cdga) over $\R$, whose underlying filtration is induced by a proper dg ideal $I$:
	\[
	F^p \uR = 
	\begin{cases}
		I^p, & p \geq 1,\\
		\uR, & p \leq 0.
	\end{cases}
	\]
	
	\begin{defn}
		A \emph{curved $L_\infty$ algebra} over $R$ is a filtered graded $\uR$-module $\fg$ together with
		\begin{itemize}
			\item an element of degree $1$
			\[
			l_0 \in F^1 \fg[1];
			\]
			\item a filtration-preserving $\R$-linear map of degree $1$
			\[
			l_1\colon \fg[1] \rightarrow \fg[1];
			\]
			\item a graded symmetric filtration-preserving multi-$\uR$-linear map of degree $1$
			\[
			l_n\colon \fg[1]^{\times n} \longrightarrow \fg[1]
			\]
			for each $n \ge 2$.
		\end{itemize}
		These multi-brackets satisfy the following identities:
		\begin{itemize}
			\item for every $r \in R^\sharp$ and $x \in \fg[1]$,
			\[
			l_1(r x)
			=
			D_R(r)\, x
			+
			(-1)^{|r|} r\, l_1(x);
			\]
			\item for every $n \ge 0$ and  $x_1,\dots,x_n \in \fg[1]$,
			\[
			\sum_{p+q=n}
			\sum_{\sigma \in \Sigma_n}
			\frac{\epsilon(\sigma)}{p!\, q!}
			\,
			l_{q+1}\big(
			l_p(x_{\sigma(1)},\dots,x_{\sigma(p)}),
			x_{\sigma(p+1)},\dots,x_{\sigma(n)}
			\big)
			=
			0,
			\]
			where $\epsilon(\sigma)$ is the Koszul sign of the permutation $\sigma$. In particular, one has 
			\[
			l_1(l_0) = 0, \quad \text{and} \quad l_2(l_0, x) + l_1(l_1(x))=0.
			\]
		\end{itemize}
		%Such $\fg$ is called \emph{curved} if $l_0 \neq 0$.
		%We sometimes write $\fg_R$ to indicate that $\fg$ is a curved $L_\infty$ algebra over $(R, D_R)$.
	\end{defn}	
	
	Since $l_0 \in F^1\fg[1]$, the pair $(\fg[1], l_1)$ defines a curved complex with $\uR$-linear curvature 
	\[
	r\coloneqq-l_2(l_0, \cdot);
	\]
	see Appendix \ref{sec:hpt} for our notion of curved complexes.
	
	\begin{defn}
		A \emph{morphism} $\phi\colon \fg_1 \rightarrow \fg_2$ between curved $L_\infty$ algebras over $R$ consists of 
		\begin{itemize}
			\item an element of degree $0$
			\[
			\phi_0 \in F^1 \fg_2[1];
			\]
			\item a graded symmetric filtration-preserving multi-$\uR$-linear map of degree $0$
			\[
			\phi_n\colon \fg_1[1]^{\times n} \rightarrow \fg_2[1]
			\]
			for each $n \geq 1$.
		\end{itemize}
		These multi-linear maps intertwine the multi-brackets on $\fg_1$ and $\fg_2$: 
		\begin{align*}
			&\sum_{p+q=n} \sum_{\sigma \in \Sigma_n}\frac{\epsilon(\sigma)}{p!q!} \,
			\phi_{q+1}\big(
			l_p(x_{\sigma(1)},\dots),\dots,x_{\sigma(n)}
			\big) \\
			&= \sum_{k=0}^{\infty} \sum_{n_1+\cdots+n_k=n} \sum_{\sigma\in \Sigma_n} \frac{\epsilon(\sigma)}{k!n_1!\cdots n_k!}
			\, l_k \big(
			\phi_{n_1}(x_{\sigma(1)},\dots),
			\dots,
			\phi_{n_k}(\dots,x_{\sigma(n)})
			\big)
		\end{align*}
		for every $n \ge 0$ and  $x_1,\dots,x_n \in \fg_1[1]$. In particular, one has 
		\begin{align*}
			&\phi_1(l_0) =  l_0 + \sum_{k=1}^\infty \frac{1}{k!}l_k(\phi_0, \dots, \phi_0), \\
			&\phi_1(l_1(x)) - l_1(\phi_1(x))
			=
			-\phi_2(l_0,x) 
			+ \sum_{k = 1}^\infty \frac{1}{k!}\, l_{k+1}(\phi_0,\dots, \phi_0, \phi_1(x)).
		\end{align*}
		Such $\phi$ is called \emph{strict} if $\phi_n = 0$ for all $n \neq 1$.
	\end{defn}
	\begin{rmk}
		The above notion of morphisms easily extends to morphisms of curved $L_\infty$ algebras over different base filtered cdgas via completed base change.
	\end{rmk}
	\begin{rmk}
		The composition $\phi \circ \psi$ of two morphisms $\psi \colon \fh \to \fg$ and $\phi \colon \fg \to \fm$ is given by
		\[
		(\phi \circ \psi)_n(x_1, \dots, x_n)
		=
		\sum_{k=0}^{\infty}
		\sum_{n_1 + \cdots + n_k = n}
		\sum_{\sigma \in \Sigma_n}
		\frac{\epsilon(\sigma)}{k!n_1! \cdots n_k!}
		\, \phi_k\!\Big(
		\psi_{n_1}(x_{\sigma(1)}, \dots), \dots, \psi_{n_k}(\dots, x_{\sigma(n)})
		\Big).
		\]
	\end{rmk}

	Since $l_0 \in F^1 \fg_1[1]$ and $\phi_0 \in F^1 \fg_2[1]$, we have
	\[
	\phi_1(l_1(x)) - l_1(\phi_1(x)) \in F^{p+1} \fg_2[1]
	\]
	for any $x \in F^p \fg_1[1]$. Thus, 
	\[
	\phi_1\colon (\fg_1[1],l_1) \rightarrow (\fg_2[1], l_1)
	\]
	defines a morphism of curved complexes.

	\begin{defn}
		A morphism $\phi=\{\phi_n\}_{n=0}^\infty$ between curved $L_\infty$ algebras $\fg_1$ and $\fg_2$  is called a \emph{weak equivalence} if 
		\[
		\Gr\, \phi_1 \colon (\Gr\, \fg_1[1], \Gr\, l_1) \longrightarrow(\Gr\, \fg_2[1], \Gr\, l_1)
		\]
		is a quasi-isomorphism.
	\end{defn}
	
	A \emph{filtered cocommutative dg coalgebra} $C=(\uC,D_C)$ over $R$ is a filtered dg module $C$ over $R$ equipped with a compatible graded cocommutative coalgebra structure on $\uC$ over $\uR$; that is, a coproduct and a counit
	\[
	\Delta \colon \uC \longrightarrow \uC \otimes_\uR \uC,
	\qquad
	\epsilon \colon \uC \longrightarrow \uR,
	\]
	compatible with both the differentials and the filtrations. Explicitly,
	\[
	(D_C \otimes \id + \id \otimes D_C)\circ \Delta
	=
	\Delta \circ D_C,
	\qquad
	D_R \circ \epsilon
	=
	\epsilon \circ D_C,
	\]
	and 
	\[
	\Delta(F^p \uC) \subset \sum_{i+j=p} F^i \uC \otimes_\uR F^j \uC,
	\qquad
	\epsilon(F^p \uC) \subset F^p \uR.
	\]
	%For simplicity, we write $C/R$ to denote a filtered cocommutative dg coalgebra over $(R,D_R)$.
	
	A \emph{morphism} between filtered cocommutative dg coalgebras $C_1$ and $C_2$ over $R$ is a morphism of graded cocommutative coalgebras over $\uR$ 
	\[
	\Phi\colon C_1^\sharp \rightarrow C_2^\sharp
	\]
	compatible with the differentials and filtrations; that is,
	\[
	\Phi(F^p \uC_1) \subset F^p \uC_2, \qquad \Phi \circ D_{C_1} = D_{C_2} \circ \Phi.
	\]	
	
	Let $\fg$ be a curved $L_\infty$ algebra over $R$. We denote by
	\[
	\wSym_\uR(\fg[1])
	\]
	the completion of the graded commutative algebra $\Sym_\uR(\fg[1])$ with respect to the convolution filtration
	\[
	F^p \Sym_\uR(\fg[1]) \coloneqq \sum_{r+s=p} F_\fg^r\, \Sym_\uR(\fg[1]) \odot F_S^{s}\, \Sym_\uR(\fg[1]),
	\]
	where $\odot$ denotes the symmetric product, the filtration $F_S^{s}\, \Sym_\uR(\fg[1])$ is defined by
	\[
	F^{s}_S\, \Sym_\uR(\fg[1]) \coloneqq \bigoplus_{i=0}^{-s} \Sym_\uR^i(\fg[1]),
	\]
	and the filtration $F_\fg^r\, \Sym_\uR(\fg[1])$ is determined by
	\[
	F^r_\fg\, \Sym^0_\uR(\fg[1]) \coloneqq F^r R, \qquad F^r_\fg\, \Sym^1_\uR(\fg[1]) \coloneqq F^r \fg[1].
	\]
	
	The coproduct and counit on $\wSym_\uR(\fg[1])$ are induced by the standard graded cocommutative coalgebra structure on $\Sym_\uR(\fg[1])$, namely, for $r \in \uR$ and $x_1,\ldots,x_n \in \fg[1]$,
	\[
	\Delta(x_1 \odot \cdots \odot x_n)
	=
	\sum_{p=0}^n
	\sum_{\sigma \in \Sigma_n}
	\frac{\epsilon(\sigma)}{p!(n-p)!}\,
	(x_{\sigma(1)} \odot \cdots \odot x_{\sigma(p)})
	\otimes
	(x_{\sigma(p+1)} \odot \cdots \odot x_{\sigma(n)}),
	\]
	and
	\[
	\epsilon(r)=r,
	\qquad
	\epsilon(x_1 \odot \cdots \odot x_n)=0.
	\]
	
	The multi-brackets $\{l_n\}_{n=0}^\infty$ of $\fg$ induce a differential $D_\fg$ on $\wSym_\uR(\fg[1])$ by
	\[
	D_\fg(r) = D_R(r),
	\]
	and
	\[
	D_\fg(x_1 \odot \cdots \odot x_n) = \sum_{p=0}^n
	\sum_{\sigma \in \Sigma_n}
	\frac{\epsilon(\sigma)}{p!(n-p)!}
	\,
	l_p(x_{\sigma(1)}, \dots, x_{\sigma(p)}) \odot
	x_{\sigma(p+1)} \odot \cdots \odot x_{\sigma(n)}.
	\]
	Since $l_0 \in F^1 \fg[1]$, $D_\fg$ is compatible with the filtration on $\wSym_\uR(\fg[1])$.  It is also compatible with the coproduct and counit of $\wSym_\uR(\fg[1])$. 
	\begin{defn}
		We call the filtered cocommutative dg coalgebra over $R$
		\[
		\CE(\fg) \coloneqq \bigl(\wSym_\uR(\fg[1]), D_\fg\bigr)
		\]
		the \emph{Chevalley--Eilenberg coalgebra} of $\fg$.\footnote{Note that the filtration used here for $\CE(\fg)$ differs from that in the pro-nilpotent setting \cite{getzler2025higherholonomycurvedlinftyalgebras}.} The differential $D_\fg$ is called the \emph{Chevalley--Eilenberg differential} of $\fg$.
	\end{defn}

	A morphism $\phi\colon \fg_1 \to \fg_2$ of curved $L_\infty$ algebras over $R$ induces a morphism of filtered cocommutative dg coalgebras over $R$
	\[
	\CE(\phi)\colon \CE(\fg_1) \longrightarrow \CE(\fg_2)
	\]
	via the formula
	\[
	\CE(\phi)(x_1 \odot \cdots \odot x_n)
	=
	\sum_{k=0}^{\infty}
	\sum_{n_1 + \cdots + n_k = n}
	\sum_{\sigma \in \Sigma_n}
	\frac{\epsilon(\sigma)}{k!n_1!\cdots n_k!}\,
	\phi_{n_1}(x_{\sigma(1)},\ldots) \odot \cdots \odot
	\phi_{n_k}(\ldots,x_{\sigma(n)}).
	\]
	The functor $\CE$ embeds the category of curved $L_\infty$ algebras over $R$ as a full subcategory of the category of filtered cocommutative dg coalgebras over $R$. 
	
	\begin{defn}\label{bc}
		Let $\fg$ be a curved $L_\infty$ algebra over $R$. Let $S=(S^\sharp,D_S)$ be a filtered cdga over $R$. 
		The \emph{completed base change} of $\fg$ along $S$ is the curved $L_\infty$ algebra over $S$ given by the completed tensor product of filtered graded $\uR$-modules
		\[
		\fg_S\coloneqq
		\uS \,\widehat{\otimes}_\uR\, \fg,
		\]
		whose multi-brackets are defined on homogeneous tensors by
		\[
		l_n(\alpha_1 \otimes x_1,\ldots,\alpha_n \otimes x_n)
		=
		\begin{cases}
			1 \otimes l_0,
			& n=0, \\[6pt]
			D_S(\alpha_1)\otimes x_1
			+
			(-1)^{|\alpha_1|}
			\alpha_1 \otimes l_1(x_1),
			& n=1, \\[6pt]
			(-1)^{\sum_{i<j}|x_i||\alpha_j|}
			\,\alpha_1\cdots\alpha_n
			\otimes
			l_n(x_1,\ldots,x_n),
			& n\ge2.
		\end{cases}
		\]
	\end{defn}	
	
	For our purposes, we consider a special class of curved $L_\infty$ algebras $\fg$ over $R$ satisfying the following condition:
	\begin{itemize}
		\item $\fg$ is \emph{finitely generated projective} over $\uR$; that is, $\fg$ is a direct summand of a free graded $\uR$-module of finite total rank in the category of filtered graded $\uR$-modules.
	\end{itemize}
	Note this condition implies that the filtration on $\fg$ is \emph{strict}, i.e.
	\[
	F^p \fg = F^p \uR \cdot \fg.
	\]
	It follows that the multi-$\uR$-linearity of the multi-brackets $l_n$ on $\fg$ automatically implies compatibility with the filtration. Moreover, the ordinary tensor product $\uS \otimes_\uR \fg$ appearing in Definition \ref{bc} is automatically complete, so that
	\[
	\fg_S = \uS \otimes_\uR \fg,
	\]
	and is again finitely generated projective over $\uS$.
	%The Chevalley--Eilenberg coalgebra of $\fg_S$ can be characterized as
	%\[
	%\CE(\fg_S) \cong \big(\uS \otimes_\uR \CE(\fg)^\sharp,\; D_S \otimes 1 + 1 \otimes D_\fg\big).
	%\]
	
	We denote by $\mathbf{L_\infty Alg}(R)_{\mathrm{fgp}}$ the category of finitely generated projective curved $L_\infty$ algebras over $R$.
	
	\subsection{Curved \texorpdfstring{$L_\infty$}{L-infinity}  spaces}
	
	In \cite{cattaneojiang26}, we adapt Costello’s notion of $L_\infty$ spaces over smooth manifolds together with the notion of weak equivalences between them \cite{costello2011geometric} to the setting of dg manifolds. We briefly recall the setup.
	
	Let $\sM=(M, \sO_\sM)$ be a dg manifold. Let $\wOmega_\sM$ denote the sheaf of completed differential forms of $\sM$ on $M$.  $\wOmega_\sM$ carries an $\sI$-adic filtration induced by the ideal of completed differential forms whose projection to functions vanishes
	 \[
	\sI \coloneqq \wOmega_\sM^{>0}.
	\]
	It also carries a differential $D_\sM$ preserving $\sI$, defined by
	\[
	D_\sM  = d_{dR} + L_{Q_\sM},
	\]
	where $d_{dR}$ is the de Rham differential and $L_{Q_\sM}$ is the Lie derivative along the cohomological vector field $Q_\sM$ of $\sM$. For simplicity, we write
	\[
	\sR_\sM \coloneqq (\wOmega_\sM, D_\sM), \qquad R_\sM\coloneqq \sR_\sM(M).
	\]
	%We call the dg ringed space 
	%\[
	%\sM_{dR}\coloneqq(M, \sR_\sM)
	%\]
	%the \emph{de Rham space} of $\sM$.

	\begin{defn}
		A \emph{curved $L_\infty$ space} over $\sM$ is a pair $B\fg=(\sM,\fg)$, where $\fg$ is a sheaf of graded $\sR_\sM^\sharp$-modules, locally free of finite total rank, equipped with a curved $L_\infty$ algebra structure over $\sR_\sM$.
	\end{defn}
	
	Let $V \subset U$ be two open subsets of $M$. By definition, the restriction map 
	\[
	\operatorname{res}^U_V \colon \fg(U) \longrightarrow \fg(V)
	\]
	is required to be a strict morphism of curved $L_\infty$ algebras over $\sR_\sM(U)$. That is, we have
	\[
	l_n\bigl(\operatorname{res}^U_V(x_1), \dots, \operatorname{res}^U_V(x_n)\bigr)
	=
	\operatorname{res}^U_V\bigl(l_n(x_1, \dots, x_n)\bigr)
	\]
	for all $x_1, \dots, x_n \in \fg(U)[1]$ and all $n \ge 0$.
	 
	\begin{defn}
		A \emph{morphism} $\phi\colon B\fg_1 \to B\fg_2$ of curved $L_\infty$ spaces is a pair $\phi=(f, \phi^\sharp)$, where 
		\begin{itemize}
			\item $f\colon \sM_1 \rightarrow \sM_2$ is a dg manifold morphism;
			\item $\phi^\sharp\colon \fg_1 \rightarrow f^*\fg_2$ is a morphism of curved $L_\infty$ algebras over $\sR_{\sM_1}$, where
			\[
			f^*\fg_2 = \sR_{\sM_1} \otimes_{f^{-1} \sR_{\sM_2}} \fg_2.
			\]
		\end{itemize}
		
		Such $\phi$ is called \emph{base-fixing} if $f=\id$. It is called \emph{strict} if $\phi^\sharp$ is strict. It is called a \emph{weak equivalence} if $f$ is a weak equivalence of dg manifolds and
		\[
		\phi^\sharp_{M_1}\colon \fg_1(M_1) \rightarrow f^*\fg_2(M_1)
		\]
		is a weak equivalence of curved $L_\infty$ algebras over $R_{\sM_1}$.
	\end{defn}
	\begin{rmk}
		In \cite{cattaneojiang26}, we formulate morphisms of curved $L_\infty$ spaces over dg manifolds using the language of Chevalley--Eilenberg algebras. The two approaches are equivalent, since one can show that $\fg$ is finitely generated and projective over $\sR_\sM^\sharp$. 
	\end{rmk}
	
	We denote by $\mathbf{L_\infty Sp}(\sM)$ the category of curved $L_\infty$ spaces over $\sM$ with base-fixing morphisms. In the next section, we show that the global sections functor
	\[
	\Gamma\colon \mathbf{L_\infty Sp}(\sM) \longrightarrow \mathbf{L_\infty Alg}(R_\sM)_{\mathrm{fgp}}
	\]
	is well-defined and fully faithful.
	
	\section{Global sections functors}
	
	We begin by reviewing the global sections functor in the general setting of graded ringed spaces, following \cite{morye2013note}.\footnote{The exposition in \cite{morye2013note} is formulated for (locally) ringed spaces; its extension to the graded ringed setting considered here is straightforward.} 
	
	Let $(X, \sO_X)$ be a graded ringed space. Let $\sO_X\text{-}\mathbf{mod}$ denote the category of graded (left) $\sO_X$-modules. Let $A$ denote the graded ring of global sections $\sO_X(X)$, and $A\text{-}\mathbf{mod}$ denote the category of graded (left) $A$-modules. We have a left exact functor
	\[
	\Gamma\colon \sO_X\text{-}\mathbf{mod} \longrightarrow A\text{-}\mathbf{mod},
	\]
	called the \emph{global sections functor}, which sends each graded $\sO_X$-module $\sE$ to its global sections $\Gamma(\sE)\coloneqq\sE(X)$. We also have a right exact functor
	\[
	\sh\colon A\text{-}\mathbf{mod} \longrightarrow \sO_X\text{-}\mathbf{mod},
	\]
	which sends each graded $A$-module $E$ to the sheafification of the presheaf $\psh(E)$ given by
	\[
	U \longmapsto \sO_X(U) \otimes_A E.
	\]
	One can show that $\sh$ is left adjoint of $\Gamma$. The natural isomorphism
	\[
	\alpha_{E,\sF}:
	\Hom(\sh(E),\sF)
	\xlongrightarrow{\cong}
	\Hom(E,\Gamma(\sF))
	\]
	is given by
	\[
	(\phi\colon \sh(E) \rightarrow \sF)\longmapsto \bigl(E \rightarrow \sh(E)(X) \xrightarrow{\phi_X} \Gamma(\sF)\bigr),
	\]
	where $E = \psh(E)(X) \rightarrow \sh(E)(X)$ denotes the canonical morphism induced by sheafification.
	
	%The key observation is that
	\begin{lem}
		The restriction of $\sh$ to the subcategory of finitely generated projective graded $A$-modules 
		\[
		\mathbf{Fgp}(A) \subset A\text{-}\mathbf{mod}
		\]
		is fully faithful.
	\end{lem}
	\begin{proof}
		Since tensoring with a finitely generated projective module commutes with projective limits, we see that $\psh(E)$ is already a sheaf for $E \in \mathbf{Fgp}(A)$. Hence $\sh(E) = \psh(E)$. 
		
		It follows that
		%$S \dashv \Gamma$ 
		\[
		\Hom(E, F) \cong \Hom(E, \Gamma(\sh(F))) \cong \Hom(\sh(E), \sh(F))
		\]
		for $E, F \in \mathbf{Fgp}(A)$.
	\end{proof} 
	
	We henceforth impose the following assumptions on the ringed space $(X,\sO_X)$:
	\begin{enumerate}[label=\textup{(\roman*)}]
		\item \label{a1} $X$ is a manifold and $\sO_X$ is a fine sheaf on $X$;
		\item \label{a2} for every surjective morphism
		\[
		\rho:\sE \longrightarrow \sF
		\]
		in the subcategory of locally free graded $\sO_X$-modules of finite total rank
		\[
		\mathbf{Vec}(X) \subset  \sO_X\text{-}\mathbf{mod},
		\]
		the kernel $\ker(\rho)$ is again in $\mathbf{Vec}(X)$.
	\end{enumerate}
	We will prove that the graded ringed space $(M, \wOmega_\sM)$ underlying a curved $L_\infty$ space $B\fg = (\sM, \fg)$ satisfies assumptions \ref{a1} and \ref{a2}.
	
	\begin{lem}\label{gff}
	    $\Gamma$ is fully faithful. Moreover, it satisfies the following conditions:
		\begin{itemize}
			\item $\sh(\Gamma(\sE)) \cong \sE$ for all $\sE \in \sO_X\text{-}\mathbf{mod}$;
			\item the restriction of $\Gamma$ to $\mathbf{Vec}(X)$	takes values in $\mathbf{Fgp}(A)$.
		\end{itemize}
	\end{lem}
	\begin{proof}
		The isomorphism $\sh(\Gamma(\sE)) \cong \sE$ can be proved using only assumption \ref{a1} on $(X,\sO_X)$. See Lemma 2.3 and Proposition 2.5 in \cite{morye2013note}. It follows that
		\[
		\Hom(\sE, \sF) \cong \Hom(\sh(\Gamma(\sE)), \sF) \cong \Hom(\Gamma(\sE), \Gamma(\sF)).
		\]
		
		We now prove the second condition using an argument due to Sardanashvily \cite{sardanashvily2001remark}. Observe that any open covering $\{U_\alpha\}$ of a manifold admits a refinement $\{U_{\beta i}\}$, where $\beta$ runs over a finite set and $i$ runs over a countable set, such that
		\[
		U_{\beta i} \cap U_{\beta j} = \emptyset \quad \text{for } i \neq j.
		\]
		Let $\sE$ be a locally free graded $\sO_X$-module of finite total rank. Let $\{U_\alpha\}$ be an open cover of $X$ over which $\sE$ trivializes. The above refinement $\{U_{\beta i}\}$ then induces a finite collection of open sets
		\[
		U_\beta \coloneqq \bigcup_i U_{\beta i}
		\]
		over which $\sE$ still trivializes.	It then follows from the partition of unity argument that there exists a finite sequence of integers $\{d_1, \dots,d_N\}$  and a surjective morphism
		\[
		\rho\colon \bigoplus_{i=1}^N\sO_X[d_i] \longrightarrow \sE.
		\]
		By assumption \ref{a2}, this fits into a short exact sequence of locally free graded $\sO_X$-modules of finite rank
		\[
		0 \longrightarrow \ker \rho \longrightarrow \bigoplus_{i=1}^N\sO_X[d_i]  \longrightarrow \sE \longrightarrow 0.
		\]
		Since all terms are fine and locally free of finite total rank, this short exact sequence splits. It follows that $\sE$ is a direct summand of a free graded $\sO_X$-module of finite total rank. Consequently, $\Gamma(\sE)$ is a finitely generated projective graded $A$-module.
 	\end{proof}
	
	Let $\mathbf{Fgp}(A)_{\mathrm{vec}}$ denote the full subcategory of $ \mathbf{Fgp}(A)$ consisting of graded $A$-modules $E$ that are isomorphic to $\Gamma(\sE)$ for some $\sE$. The two lemmas above combine to yield the following result.
	
	\begin{prop}\label{gft}
		The functors
		\[
		\Gamma\colon \mathbf{Vec}(X) \longrightarrow \mathbf{Fgp}(A)_{\mathrm{vec}},
		\qquad
		\sh\colon \mathbf{Fgp}(A)_{\mathrm{vec}} \longrightarrow \mathbf{Vec}(X)
		\]
		are well-defined and quasi-inverse equivalences.
	\end{prop}
	If $(X,\sO_X)$ is locally ringed, one can further show that the restriction of $\sh$ to $\mathbf{Fgp}(A)$ takes values in $\mathbf{Vec}(X)$, since every projective module over a local ring is free.  In this case, $\Gamma$ and $\sh$ define quasi-inverse equivalences
	\[
	\Gamma:\mathbf{Vec}(X)\longrightarrow \mathbf{Fgp}(A),
	\qquad
	\sh:\mathbf{Fgp}(A)\longrightarrow \mathbf{Vec}(X).
	\]
	This is a version of the Serre--Swan theorem. 
	However, the graded ringed space $(M, \wOmega_\sM)$ underlying a curved $L_\infty$ space $B\fg = (\sM, \fg)$ is not, in general, locally ringed.
	
	\begin{prop}\label{kerOmega}
		The graded ringed space $(M, \wOmega_\sM)$ satisfies the assumptions \ref{a1} and \ref{a2}.
	\end{prop}
	To verify assumption \ref{a2}, we need the following lemma, which is proved in \cite{cattaneojiang26}.
	\begin{lem}[Proposition 2.6 in \cite{cattaneojiang26}]\label{kerOM}
		Let $\sE_1$ and $\sE_2$ be locally free graded $\sO_\sM$-modules of finite total rank. Let $\rho\colon \sE_1 \to \sE_2$ be a surjective morphism of graded $\sO_\sM$-modules. Then $\ker \rho$ is again locally free of finite total rank.
	\end{lem}
	Recall that $\wOmega_\sM$ is defined as
	\[
	\wOmega_\sM = \wSym_{\sO_\sM}(\sT_\sM[1]^\vee),
	\]
	where $\sT_\sM$ is the tangent sheaf of $\sM$ and $ \wSym_{\sO_\sM}$ is the completed symmetric power of graded $\sO_\sM$-modules. In particular, every graded $\wOmega_\sM$-module is a graded $\sO_\sM$-module.
	\begin{lem}\label{lkerOmega}
		Let $\fg$ be a locally free graded $\wOmega_\sM$-module of finite total rank. Then there exists a locally free graded $\sO_\sM$-module $\sE$ of finite total rank such that
		\[
		\fg \cong \wOmega_\sM \otimes_{\sO_\sM} \sE.
		\]
	\end{lem}
	\begin{proof}
		Consider the short exact sequence of graded $\sO_\sM$-modules
		\[
		0 \longrightarrow F^1 \fg \longrightarrow \fg \longrightarrow \sE \coloneqq \Gr^0 \fg \longrightarrow 0.
		\]
		Since both $F^1 \fg$ and $\sE$ are (filtered) locally free over $\sO_\sM$ and fine, this short exact sequence admits a splitting $\iota\colon \sE \to \fg$. We then obtain a graded $\wOmega_\sM$-module morphism
		\begin{align*}
			\widetilde{\iota}\colon \wOmega_\sM \otimes_{\sO_\sM} \sE &\longrightarrow \fg \\
			\alpha \otimes e &\longmapsto \alpha \cdot \iota(e).
		\end{align*}
		Since the filtration on $\wOmega_\sM$ is complete, it is straightforward to verify that $\widetilde{\iota}$ is locally an isomorphism, hence an isomorphism.
	\end{proof}
	
	\begin{proof}[Proof of Proposition \ref{kerOmega}]
		Since every graded manifold splits, $\wOmega_\sM$ is a $C^\infty_M$-module and hence a fine sheaf.
		
	    By Lemma \ref{lkerOmega}, every morphism $\rho\colon \fg \to \fh$ of locally free graded $\wOmega_\sM$-modules of finite total rank can be represented in the form
		\[
		\rho\colon \wOmega_\sM \otimes_{\sO_\sM} \sE \longrightarrow \wOmega_\sM \otimes_{\sO_\sM} \sF,
		\]
		where $\sE$ and $\sF$ are locally free graded $\sO_\sM$-modules of finite total rank. Such a morphism is uniquely determined by a family of morphisms of graded $\sO_\sM$-modules
		\[
		\rho_n\colon \sE \longrightarrow \Omega^n_\sM \otimes_{\sO_\sM} \sF, \qquad n \ge 0.
		\]
		
		If $\rho$ is surjective, then the component $\rho_0\colon \sE \rightarrow \sF$ is also surjective. By Lemma \ref{kerOM}, we see that $\rho_0$ admits a splitting $s\colon \sF \rightarrow \sE$. We then obtain a surjective morphism
		\[
		\widetilde{\rho_0} \coloneqq \id \otimes \rho_0 \colon \wOmega_\sM \otimes_{\sO_\sM} \sE \longrightarrow \wOmega_\sM \otimes_{\sO_\sM} \sF,
		\]
		and a splitting of $\widetilde{\rho_0}$
		\[
		\widetilde{s} \coloneqq \id \otimes s \colon \wOmega_\sM \otimes_{\sO_\sM} \sF \longrightarrow \wOmega_\sM \otimes_{\sO_\sM} \sE.
		\]
		By construction, $\rho - \widetilde{\rho_0}$ increases the filtration. Consequently, $\id + \widetilde{s} \circ (\rho - \widetilde{\rho_0})$ is invertible, and we have
		\[
		\widetilde{\rho_0} = \rho \circ \bigl(\id + \widetilde{s} \circ (\rho - \widetilde{\rho_0})\bigr)^{-1}.
		\]
		Therefore,
		\[
		\ker \rho \cong \ker \widetilde{\rho_0} = \wOmega_\sM \otimes_{\sO_\sM} \ker \rho_0.
		\]
		By Lemma \ref{kerOM}, we see that $\ker \rho$ is locally free over $\wOmega_\sM$ of finite total rank. 
	\end{proof}
	
	\begin{prop}\label{gsf}
		The global sections functor
		\[
		\Gamma\colon \mathbf{L_\infty Sp}(\sM) \longrightarrow \mathbf{L_\infty Alg}(R_\sM)_{\mathrm{fgp}}
		\]
		is well-defined and fully faithful. Moreover, let 
		\[
		\mathbf{L_\infty Alg}(R_\sM)_{\mathrm{vec}} \subset \mathbf{L_\infty Alg}(R_\sM)_{\mathrm{fgp}}
		\]
		denote the full subcategory consisting of those $\fg_M$ such that $\fg_M \cong \Gamma(B\fg)$ for some $B\fg=(\sM, \fg) \in \mathbf{L_\infty Sp}(\sM)$. Then the functors
		\[
		\Gamma \colon \mathbf{L_\infty Sp}(\sM) \longrightarrow \mathbf{L_\infty Alg}(R_\sM)_{\mathrm{vec}}, \qquad
		\sh \colon \mathbf{L_\infty Alg}(R_\sM)_{\mathrm{vec}} \longrightarrow \mathbf{L_\infty Sp}(\sM)
		\]
		are well-defined quasi-inverse equivalences.
	\end{prop}
	
	\begin{proof}	
		The well-definedness and fully faithfulness of $\Gamma$ follow from Proposition \ref{kerOmega} and Lemma \ref{gff}. To prove the second statement, let $\fg_M \cong \Gamma(B\fg)$ for some $B\fg$. Recall that $\sh(\fg_M)$ is defined by the assignment
		\[
		U \longmapsto \sR_\sM(U) \otimes_{\sR_\sM(M)} \fg_M.
		\]
		By our previous discussion, $\sh(\fg_M)$ is the completed base change of $\fg_M$ along the restriction map, hence a sheaf of curved $L_\infty$ algebras over $\sR_\sM$.
		
		By Proposition \ref{gft}, $\sh$ is well-defined and we have isomorphisms of graded $R_\sM^\sharp$-modules
		\[
		\Gamma(\sh(\fg_M)) \cong \fg_M, \qquad \sh(\Gamma(B\fg)) \cong B\fg,
		\]
		which are clearly strict isomorphisms of curved $L_\infty$ algebras.
	\end{proof}

	\section{Category of fibrant objects}
	
	A \emph{category with weak equivalences} is a category $\bfC$ together with a subcategory
	\[
	\operatorname{core}(\bfC) \subseteq \bfW \subseteq \bfC,
	\]
	where $\operatorname{core}(\bfC)$ denotes the maximal subgroupoid of $\bfC$. Morphisms in $\bfW$ are called \emph{weak equivalences} and are denoted by $\xrightarrow{\sim}$. We require that weak equivalences satisfy the 2-out-of-3 property: if $f$ and $g$ are composable morphisms in $\bfC$, and two of $f$, $g$, and $g \circ f$ are weak equivalences, then so is the third.

	The notion of a category of fibrant objects was introduced by Brown as a flexible alternative to Quillen’s model categories for doing homotopy theory \cite{brown1973abstract}.
	
	\begin{defn}
		A \emph{category of fibrant objects (CFO)} is a category with weak equivalences $\bfC$ together with an additional subcategory
		\[
		\operatorname{core}(\bfC) \subseteq \bfF \subseteq \bfC.
		\]
		Morphisms in $\bfF$ are called \emph{fibrations} and are denoted by $\twoheadrightarrow$. Morphisms in $\bfF \cap \bfW$ are called \emph{trivial fibrations} and are denoted by $\xtwoheadrightarrow{\sim}$. These data satisfy the following axioms:
		\begin{itemize}
			\item $\bfC$ has a terminal object, and every terminal morphism is a fibration.
			
			\item Pullbacks of (trivial) fibrations exist and are again (trivial) fibrations.
			
			\item For every object $X \in \bfC$, there exists a factorization 
			\[
			\begin{tikzcd}[row sep=10pt,column sep=40pt]
				& PX \arrow[dr, two heads, "{(\pr_0,\pr_1)}"] & \\
				X \arrow[ur, "\sim"] \arrow[rr] && X \times X
			\end{tikzcd}
			\]
			of the diagonal morphism, where $PX$ is called the \emph{path space} of $X$.
		\end{itemize}
	\end{defn}
	
	Let $\bfC$ be a category of fibrant objects. For an object $X \in \bfC$, define simplicial objects
	\[
	l_\bullet X,~ r_\bullet X \in \mathcal \bfC^{\Delta^{\mathrm{op}}}
	\]
	by setting:
	\begin{itemize}
		\item $l_n X = X$, with all face and degeneracy maps equal to $\mathrm{id}_X$;
		
		\item $r_n X =$ the $(n+1)$-fold product of $X$, with face maps given by projections and degeneracy maps given by diagonal inclusions.
	\end{itemize}
	We adapt the notion of simplicial frames from model categories to CFOs.
	% Defn 5.2.7
	
	\begin{defn}
		A \emph{simplicial frame} of an object $X \in \bfC$ is a simplicial object $X_\bullet \in \bfC^{\Delta^{\mathrm{op}}}$ together with a factorization 
		\begin{equation}\label{sf}
			\begin{tikzcd}[row sep=10pt,column sep=40pt]
				& X_\bullet \arrow[dr, two heads] & \\
				l_\bullet X  \arrow[ur, hook,"\sim"] \arrow[rr] && r_\bullet X
			\end{tikzcd}
		\end{equation}
		of the diagonal morphism, which is an isomorphism in degree $0$.
	\end{defn}
	
	The existence of simplicial frames implies the existence of path spaces: the factorization \eqref{sf} reduces to $X \xlongrightarrow{\sim} X_1 \twoheadrightarrow X \times X$ in degree $1$.
	
	\begin{defn}
		A \emph{(right) framing} of $\bfC$ is a functor
		\[
		(\cdot)_\bullet\colon \bfC \longrightarrow \bfC^{\Delta^{\mathrm{op}}}
		\]
		together with natural transformations
		\[
		l_\bullet X \longrightarrow X_\bullet  \quad\text{and}\quad X_\bullet \longrightarrow r_\bullet X,
		\]
		which make $X_\bullet$ a simplicial frame of every $X \in \bfC$. 
	\end{defn}

	Every model category admits a framing; hence its subcategory of fibrant objects inherits one. See, for instance, Theorem 5.2.8 in \cite{hovey2007model}.

	\subsection{Main results}
	
	Let $R = (\uR, D_R)$ be a cdga equipped with a complete $I$-adic filtration, where $I$ is a proper dg ideal of $R$. The category $\mathbf{L_\infty Alg}(R)_{\mathrm{fgp}}$ of finitely generated projective curved $L_\infty$ algebras over $R$ has a terminal object: the zero curved $L_\infty$ algebra over $R$.
	
	\begin{defn}
		A morphism $\phi\colon \fg \rightarrow \fm$ in $\mathbf{L_\infty Alg}(R)_{\mathrm{fgp}}$ is a \emph{fibration} if its unary component $\phi_1\colon \fg[1] \rightarrow \fm[1]$ is surjective.
	\end{defn}
	
	Clearly, every terminal morphism in $\mathbf{L_\infty Alg}(R)_{\mathrm{fgp}}$  is a fibration.
	
	\begin{prop}\label{pullback}
		Let $\phi\colon \fg \rightarrow \fm$ and $\psi\colon \fn \rightarrow \fm$ be morphisms of $L_\infty$ algebras over $R$. If $\phi$ is a fibration, then the fiber product of curved $L_\infty$ algebras 
		\[
		\begin{tikzcd}
			\fg \times_\fm \fn \arrow[r,"\widetilde{\psi}"] \arrow[d,"\widetilde{\phi}"'] & \fg \arrow[d,"\phi"] \\
			\fn \arrow[r,"\psi"'] & \fm
		\end{tikzcd}
		\]
		exists. Moreover, if $\phi$ is a trivial fibration, then $\widetilde{\phi}$ is also a trivial fibration.
		%and $\widetilde{\phi}$ can be taken to be a strict fibration. 
		%Moreover, if $\psi$ is strict, then $\widetilde{\psi}$ can also be taken to be strict.
	\end{prop}
	
	The proof is essentially the same as the one in \cite{getzler2025higherholonomycurvedlinftyalgebras}; see also \cite{rogers2020explicit}. We begin by proving the following standard lemma.
	
	\begin{lem}\label{fibstr}
		Every fibration in $\mathbf{L_\infty Alg}(R)_{\mathrm{fgp}}$ is strict up to isomorphism.
	\end{lem}
	
	\begin{proof}
		Let $\phi \colon \fg \to \fm$ be a fibration. By definition, the linear term
		\[
		\phi_1 \colon \fg[1] \to \fm[1]
		\]
		is surjective. Since $\fm$ is projective, we can find a splitting
		\[
		s \colon \fm[1] \to \fg[1]
		\]
		of $\phi_1$, and let $\fk = \ker(\phi_1)[-1]$. Then
		\[
		\fg[1] \cong \fk[1] \oplus s(\fm[1]).
		\]
		
		Consider the filtered graded cocommutative coalgebra morphism
		\[
		\Psi \colon \wSym_\uR(\fg[1]) \to \wSym_\uR(\fg[1])
		\]
		over $\uR$, whose Taylor coefficients are given by
		\[
		\Psi_n = \begin{cases}
			\id, & n=1, \\
			s \circ \phi_n, & n \neq 1.
		\end{cases}
		\]
		Since the completion on $\wSym_\uR(\fg[1])$ is complete and $s \circ \phi_n$ increases the filtration, it is easy to check that $\Psi$ defines an automorphism. On the other hand, consider the filtered graded cocommutative coalgebra morphism
		\[
		\Phi\colon \wSym_\uR(\fg[1]) \to \wSym_\uR(\fm[1])
		\]
		over $\uR$, whose Taylor coefficients are given by
		\[
		\Phi_n = \begin{cases}
			\phi_1, & n=1,\\
			0, & n \neq 1.
		\end{cases}
		\]
		We claim that $\CE(\phi) =  \Phi \circ \Psi$. Indeed, recall that the Taylor coefficients $\{(\phi \circ \psi)_n\}_{n=0}^\infty$ of $\Phi \circ \Psi$ are given by the formula
		\[
		(\phi \circ \psi)_n(x_1, \dots, x_n)
		=
		\sum_{k=0}^{\infty}
		\sum_{n_1 + \cdots + n_k = n}
		\sum_{\sigma \in \Sigma_n}
		\frac{\epsilon(\sigma)}{k!n_1! \cdots n_k!}
		\, \phi_k\!\Big(
		\psi_{n_1}(x_{\sigma(1)}, \dots), \dots, \psi_{n_k}(\dots, x_{\sigma(n)})
		\Big).
		\]
		In our case, this simplifies to
		\[
		(\phi \circ \psi)_n
		= \phi_1 \circ \psi_n
		= \begin{cases}
			\phi_1, & n=1\\
			\phi_1 \circ s \circ \phi_n, & n \neq 1
		\end{cases}
		= \phi_n.
		\]
		Using $\Psi$ to conjugate the Chevalley--Eilenberg differential on $\fg$, we obtain a strict morphism isomorphic to $\phi$.
	\end{proof}
	
	We now return to the proof of Proposition \ref{pullback}. By Lemma \ref{fibstr}, we may assume that $\phi$ is strict. Let 
	\[
	s\colon \fm[1] \rightarrow \fg[1] \quad \text{and} \quad \fk = \ker \phi_1
	\]
	be as in the proof of Lemma \ref{fibstr}. Let $p\colon \fg[1] \rightarrow \fg[1]$ denote the idempotent operator 
	\[
	\id - s \circ \phi_1
	\]
	on $\fg[1]$. The fiber product $\fg \times_\fm \fn$ is realized on the filtered graded $\uR$-module $\fk \oplus \fn$. We take the morphism
	\[
	\widetilde{\phi} \colon \fk \oplus \fn \longrightarrow \fn
	\]
	to be the strict fibration given by projection onto the second factor. The morphism
	\[
	\widetilde{\psi} \colon \fk \oplus \fn \longrightarrow \fg
	\]
	is then required to satisfy
	\[
	\begin{aligned}
		\phi_1\left( \widetilde{\psi}_n(\zeta_1, \dots, \zeta_n) \right)
		=
		\psi_n\left( y_1, \dots, y_n \right)
	\end{aligned}
	\]
	for $\zeta_i = (x_i, y_i) \in \fk \oplus \fn$. These equations admit a unique solution satisfying the conditions
	\begin{align*}
		p \circ \widetilde{\psi}_{n}  = 
		\begin{cases}
			p, & n=1,\\
			0, & n\neq 1.
		\end{cases}
	\end{align*}
	given by
	\begin{equation}\label{solpsi}
		\begin{aligned}
			\widetilde{\psi}_n(\zeta_1, \dots, \zeta_n) =
			\begin{cases}
				x_1 + s \circ \psi_1(y_1), & n=1,\\
				s \circ \psi_n(y_1, \dots, y_n) & n\neq 1.
			\end{cases}
		\end{aligned}
	\end{equation}
	
	The multi-brackets $\{l_n = (\pi_1 \circ l_n, \pi_2 \circ l_n)\}_{n=0}^\infty$ on the fiber product $\mathfrak{g} \times_{\mathfrak{m}} \mathfrak{n}$ are characterized by their compatibility with $\widetilde{\phi}$ and $\widetilde{\psi}$, where $\pi_1\colon \fk \oplus \fn \rightarrow \fk \subset \fg$ and $\pi_2\colon \fk \oplus \fn \rightarrow \fn$ are the canonical projections. Explicitly, compatibility with $\widetilde{\phi}$ yields the characterization
	\begin{equation}\label{compphi}
		\pi_2 \circ l_n(\zeta_1,\ldots,\zeta_n)
		=
		l_n(y_1, \dots, y_n).
	\end{equation}
	Compatibility with $\widetilde{\psi}$ means that
	\begin{equation}\label{pullback-morphism}
		\begin{aligned}
			&
			\widetilde{\psi}_1
			\bigl(
			l_n(\zeta_1,\ldots,\zeta_n)    
			\bigr)
			+
			\sum_{\sigma\in \Sigma_n}\sum_{k=0}^{n-1}
			\frac{\epsilon(\sigma)}{k!(n-k)!}
			\widetilde{\psi}_{n-k+1}
			\bigl(
			l_k(\zeta_{\sigma(1)},\ldots,\zeta_{\sigma(k)}),
			\zeta_{\sigma(k+1)},\ldots,\zeta_{\sigma(n)}
			\bigr)
			\\
			&=
			\sum_{k=0}^{\infty}
			\frac{1}{k!}
			\sum_{\sigma\in \Sigma_n}\sum_{n_1+\cdots+n_k=n}
			\frac{\epsilon(\sigma)}{n_1!\cdots n_k!}
			l_k
			\Bigl(
			\widetilde{\psi}_{n_1}(\zeta_{\sigma(1)},\ldots),
			\ldots,
			\widetilde{\psi}_{n_k}(\ldots,\zeta_{\sigma(n)})
			\Bigr).
		\end{aligned}
	\end{equation}
	Applying \eqref{solpsi} and \eqref{compphi} to \eqref{pullback-morphism} yields the characterization
	\[
	\begin{aligned}
		\pi_1 \circ l_n(\zeta_1,\ldots,\zeta_n) &=
		\sum_{k=0}^{\infty}
		\frac{1}{k!}
		\sum_{\sigma\in \Sigma_n}\sum_{n_1+\cdots+n_k=n}
		\frac{\epsilon(\sigma)}{n_1!\cdots n_k!}
		l_k
		\Bigl(
		\widetilde{\psi}_{n_1}(\zeta_{\sigma(1)},\ldots),
		\ldots,
		\widetilde{\psi}_{n_k}(\ldots,\zeta_{\sigma(n)})
		\Bigr) \\
		&-\sum_{\sigma\in \Sigma_n}\sum_{k=0}^n
		\frac{\epsilon(\sigma)}{k!(n-k)!}
		s \circ \psi_{n-k+1}
		\bigl(
		l_k(y_{\sigma(1)},\ldots,y_{\sigma(k)}),
		y_{\sigma(k+1)},\ldots, y_{\sigma(n)}
		\bigr).
	\end{aligned}
	\] 
	
	We need to verify that the multi-brackets on $\fg \times_\fm \fn$ are well defined.
	%For $n=0$ and $n=1$, we have $l_0 = (p(l_0), l_0) \in \fk \oplus \fn$ and
	The unary bracket $l_1$ is given by
	\begin{equation}\label{l1}
		\begin{aligned}
			&\pi_1 \circ l_1(\zeta_1) =\sum_{k=0}^\infty\frac{1}{k!}l_{k+1}(s(\psi_0), \dots, s(\psi_0), x_1 + s (\psi_1(y_1)))- s\left( \psi_1(l_1(y_1)) + \psi_2(l_0,y_1)\right), \\
			&\pi_2 \circ l_1(\zeta_1) =l_1(y_1). 
		\end{aligned}
	\end{equation}
	%\begin{align*}
	%	l_1(x_1 + s (\psi_1(y_1)))- s\left( \psi_1(l_1(y_1))\right),
	%\end{align*}
	It is easy to see that $l_1$ satisfies the Leibniz rule over $R$, i.e.
	\[
	l_1(r \zeta_1) - (-1)^{|r|} r l_1(\zeta_1) = D_R(r) \zeta_1.
	\]
	
	Let $D$ denote the degree $1$ coderivation on $\wSym_\uR((\fg \times_\fm \fn)[1])$ induced by the multi-brackets on $\fg \times_\fm \fn$. By construction,
	\[
	D^2\bigl(\wSym_\uR((\fg \times_\fm \fn)[1])\bigr)
	\]
	lies in the kernels of both $\CE(\widetilde{\phi})$ and $\CE(\widetilde{\psi})$, and hence must vanish. Therefore, the multi-brackets on $\fg \times_\fm \fn$ satisfy the strong homotopy Jacobi identities.

	To see that $\mathfrak{g}\times_{\mathfrak{m}} \mathfrak{n}$ is a pullback, consider a commutative diagram of the form
	\[
	\begin{tikzcd}[column sep=30pt, row sep=30pt]
		\mathfrak{a}
		\arrow[bend right=25,ddr,"\nu"']
		\arrow[bend left=25,drr,"\lambda"]
		\arrow[dashed,dr,"\epsilon"']
		&&
		\\
		&
		\mathfrak{g}\times_{\mathfrak{m}} \mathfrak{n}
		\arrow[d,"\widetilde{\phi}"']
		\arrow[r,"\widetilde{\psi}"]
		&
		\mathfrak{g} \arrow[d,"\phi"]
		\\
		&
		\mathfrak{n} \arrow[r,"\psi"']
		&
		\mathfrak{m}
	\end{tikzcd}
	\]
	The morphism $\epsilon$ is given by
	\[
	\epsilon_n(z_1,\ldots,z_n)
	=
	(p \circ \lambda_n(z_1,\ldots,z_n), \nu_n(z_1,\ldots,z_n)).
	\]
	This proves the universal property.
	
	We now turn to the proof of the second statement. Using \eqref{l1}, we see that
	\[
	\Gr\, l_1 =
	\begin{pmatrix}
		\Gr\, l_1 & \Gr\,[l_1, s \circ \psi_1] \\
		0 & \Gr\, l_1
	\end{pmatrix}
	\]
	on $\fk \oplus \fn$. We can then identify
	\[
	H^\bullet(\Gr\, \widetilde{\phi}_1)\colon
	H^\bullet(\Gr\, \fg \times_\fm \fn)
	\cong
	H^\bullet(\Gr\, \fk)\oplus H^\bullet(\Gr\, \fn)
	\longrightarrow
	H^\bullet(\Gr\, \fn)
	\]
	with the projection onto the second factor. Therefore, if $\phi$ is a weak equivalence, then
	\[
	H^\bullet(\Gr\, \fk)=0,
	\]
	and hence $H^\bullet(\Gr\, \widetilde{\phi}_1)$ is an isomorphism. This completes the proof of Proposition~\ref{pullback}.

	For each $n \ge 0$, consider the cdga $\Omega_n = (\Omega_n^\sharp, d)$, where
	\[
	\Omega_n^\sharp =
	\mathbb{R}[t_0,\ldots,t_n,dt_0,\ldots,dt_n]/(t_0+\cdots+t_n-1,\; dt_0+\cdots+dt_n),
	\]
	and the differential $d$ is determined by
	\[
	d(t_i)=dt_i, \qquad d(dt_i)=0.
	\]
	This cdga can be identified with the algebra of polynomial forms on the geometric $n$-simplex
	\[
	|\Delta^n|
	=
	\left\{
	(t_0,\dots,t_n)\in \mathbb{R}^{n+1}
	\;\middle|\;
	t_0 \ge 0,\; \dots,\; t_n \ge 0,\;
	\sum_{i=0}^n t_i = 1
	\right\},
	\]
	equipped with the de Rham differential.
	
	Let $\fg \in \mathbf{L_\infty Alg}(R)_{\mathrm{fgp}}$. The completed tensor product 
	\[
    \Omega_n^\sharp \,\wotimes_\R\, \fg
	\]
	inherits a natural structure of a curved $L_\infty$ algebra over $R$, with the induced filtration
	\[
	F^p(\Omega_n^\sharp\,\wotimes_\R\, \fg) = \Omega_n^\sharp \,\wotimes_\R\, F^p \fg,
	\]
	and the induced multi-brackets
	\[
	l_n(\alpha_1 \otimes x_1,\ldots,\alpha_n \otimes x_n)
	=
	\begin{cases}
		1 \otimes l_0, & n=0, \\
		d\alpha_1 \otimes x_1 + (-1)^{|\alpha_1|}\alpha_1 \otimes l_1(x_1), & n=1, \\
		(-1)^{\sum_{i<j} |x_i||\alpha_j|}\, \alpha_1 \cdots \alpha_n \otimes l_n(x_1,\dots,x_n), & n>1.
	\end{cases}\footnote{This construction should not be confused with a completed base change of $\fg$.}
	\]
	
	We would like $\Omega_\bullet^\sharp \,\wotimes_\R\, \fg$ to define a simplicial frame for $\fg$. However, the component $\Omega_n^\sharp \,\wotimes_\R\, \fg$ is not finitely generated for $n \geq 1$. To remedy this, one applies Dupont's contraction to retract $\Omega_n$ onto the finite-dimensional Whitney complex $W_n \subset \Omega_n$, which is spanned by the elementary forms
	\[
	\omega_{i_0\cdots i_k}
	=
	k!
	\sum_{j=0}^k (-1)^j\, t_{i_j}\,
	dt_{i_0}\wedge \cdots \widehat{dt_{i_j}} \cdots \wedge dt_{i_k},
	\qquad
	0 \le i_0 < \cdots < i_k \le n.
	\]
	
	Dupont's contraction takes the form \cite{dupont1976simplicial} 
	\[
	s_n
	=
	\sum_{k=0}^{n-1}
	\sum_{0 \le i_0 < \cdots < i_k \le n}
	\omega_{i_0\cdots i_k}\,
	h_n^{i_k}\cdots h_n^{i_0},
	\]
	where the operators $h_n^i$ are defined by
	\[
	h_n^i(\omega)
	=
	\int_0^1 \frac{du}{u}\,
	(\varphi_i(u))^* \iota_{E_i}\omega.
	\]
	Here $\iota_{E_i}$ denotes contraction with the $i$-th Euler vector field on $|\Delta^n|$
	\[
	E_i
	\coloneqq
	\sum_{j=0}^n (t_j - \delta_{ij})\frac{\partial}{\partial t_j},
	\]
	and
	\[
	\varphi_i \colon [0,1]\times |\Delta^n|
	\longrightarrow |\Delta^n|,
	\qquad
	(u,\mathbf{t})
	\longmapsto
	u\mathbf{t} + (1-u)e_i,
	\]
	is the dilation flow generated by $E_i$, where $e_i$ denotes the $i$-th vertex of the simplex. 
	
	Let $i_n\colon W_n \rightarrow \Omega_n$ denote the canonical inclusion map. Dupont proved that
	\[
	\id - i_n p_n = [d, s_n],
	\]
	where $p_n \colon \Omega_n \to W_n$ denotes Whitney’s projection. In terms of the contractions $h_n^i$, it can be expressed as 
	\[
	p_n
	=
	\sum_{k=0}^n (-1)^k
	\sum_{0 \le i_0 < \cdots < i_k \le n}
	\omega_{i_0\cdots i_k}\,
	\varepsilon_{i_k}^n\,
	h^{i_{k-1}}_n \cdots h^{i_0}_n,
	\]
	where $\varepsilon^n_i\colon \Omega_n \rightarrow \R$ is the evaluation map at the vertex $e_i$. 
	
	\begin{lem}
		The triple $(p_n, i_n, s_n)$ defines a strong deformation retract of (flat) complexes from $\Omega_\bullet$ to $W_\bullet$. That is, we have
		\[
		\id - i_np_n = [d,s_n],
		\qquad
		\id - p_ni_n = 0,
		\qquad
		s_n^2 = p_ns_n = s_ni_n = 0.
		\]
	\end{lem}
	The side conditions are proved in Lemma 3.4 and Theorem 3.11 of \cite{getzler2009}.
	
	\begin{lem}
		$s_n \otimes \id$ defines a contraction of the curved $L_\infty$ algebra $\Omega_n^\sharp \,\wotimes_\R\, \fg$ over $R$.
	\end{lem}
	See Appendix \ref{sec:htt} for our notion of contractions of curved $L_\infty$ algebras over $R$.
	\begin{proof}
	    We have 
	    \[
	    (s_n \otimes \id) (d \otimes \id + \id \otimes l_1) (s_n \otimes \id) = (s_n d s_n) \otimes \id + s_n^2 \otimes l_1 = s_n \otimes \id, 
	    \]
	    and
	    \[
	    [s_n \otimes \id, (d \otimes \id + \id \otimes l_1)^2]= [s_n \otimes \id, \id \otimes l_1^2] = 0.
	    \]
	    Moreover, $s_n \otimes \id$ is $\uR$-linear and $(s_n \otimes \id)(1 \otimes l_0) = 0$.
	\end{proof}
	
	Let 
	\[
	\bfi_n\colon \wSym_\uR(W_n^\sharp \otimes_\R \fg[1]) \longrightarrow \wSym_\uR(\Omega_n^\sharp \,\wotimes_\R\, \fg[1]) 
	\]
	and
	\[
	\bfp_n\colon \wSym_\uR(\Omega_n^\sharp \,\wotimes_\R\, \fg[1]) \longrightarrow \wSym_\uR(W_n^\sharp \otimes_\R \fg[1]) 
	\]
	denote the inclusion and projection of filtered graded cocommutative coalgebras over $\uR$ induced by $i_n \otimes \id$ and $p_n \otimes \id$. Let 
	\[
	\bfs_n\colon  \wSym_\uR(\Omega_n^\sharp \,\wotimes_\R\, \fg[1])  \longrightarrow \wSym_\uR(\Omega_n^\sharp \,\wotimes_\R\, \fg[1]) 
	\]
	denote the contraction induced by $s_n \otimes \id$, see \eqref{bfh}.
	
	Let $\bfd_n$ denote the curved differential on
	$\wSym_{\uR}(\Omega_n^\sharp \,\widehat{\otimes}_{\R}\, \fg[1])$
	induced by the curvature and unary bracket of
	$\Omega_n^\sharp \,\widehat{\otimes}_{\R}\, \fg[1]$. Let $D_n$ denote the Chevalley--Eilenberg differential of
	$\Omega_n^\sharp \,\widehat{\otimes}_{\R}\, \fg[1]$. As shown in Appendix \ref{sec:htt}, the perturbation
	\[
	\mu_n \coloneqq D_n - \bfd_n
	\]
	of the curved complex
	$\left(\wSym_{\uR}(\Omega_n^\sharp \,\widehat{\otimes}_{\R}\, \fg[1]), \bfd_n\right)$
	induces a curved $L_\infty$ algebra structure on
	\[
	\fg_n \coloneqq W_n^\sharp \otimes_\R \fg,
	\]
	with Chevalley--Eilenberg differential
	\[
	D_{\fg_n} \coloneqq \bfdelta_n + (\bfp_\mu)_n \, \mu_n \, \bfi_n,
	\]
	where
	\[
	\bfdelta_n = \bfp_n \, \bfd_n \, \bfi_n,
	\qquad
	(\bfp_\mu)_n = \bfp_n \, (\id + \mu_n \bfs_n)^{-1}.
	\]
	Moreover, we obtain morphisms of curved $L_\infty$ algebras over $R$
	\[
	(\bfi_\mu)_n\colon \fg_n[1] \longrightarrow \Omega_n^\sharp \,\widehat{\otimes}_{\R}\, \fg[1],
	\qquad
	(\bfp_\mu)_n\colon \Omega_n^\sharp \,\widehat{\otimes}_{\R}\, \fg[1] \longrightarrow \fg_n[1],
	\]
	which are both weak equivalences. Here,
	\[
	(\bfi_\mu)_n \coloneqq (\id + \bfs_n \mu_n)^{-1}\bfi_n.
	\]
		
	\begin{prop}\label{simfr}
		We have a factorization 
		\[
		\begin{tikzcd}[row sep=10pt,column sep=40pt]
			& \fg_n \arrow[dr, two heads] & \\
			l_n \fg  \arrow[ur, "\sim"] \arrow[rr] && r_n \fg
		\end{tikzcd}
		\]
		of the diagonal morphism $l_n \fg  \rightarrow  r_n \fg$ in $\mathbf{L_\infty Alg}(R)_{\mathrm{fgp}}$, which is an isomorphism in degree $0$.
	\end{prop}
	
	\begin{proof}
		We have the following commutative diagram
		\[
		\begin{tikzcd}[row sep=20pt, column sep=60pt]
			& \fg_n = W_n^\sharp \otimes_\R \fg
			\arrow[dr, two heads]
			& \\
			l_n \fg =\fg
			\arrow[ur, "\sim"]
			\arrow[rr]
			&&
			r_n \fg =\fg^{\oplus (n+1)}
		\end{tikzcd}
		\]
		of curved $L_\infty$-algebras over $R$. The weak equivalence
		\[
		\fg \longrightarrow W_n^\sharp \otimes_\R \fg
		\]
		is given by the composition
		\[
		\fg \xlongrightarrow{\iota_n} \Omega_n^\sharp \,\widehat{\otimes}_{\R}\, \fg
		\xlongrightarrow{(\bfp_\mu)_n}
		W_n^\sharp \otimes_\R \fg,
		\]
		where the first morphism $\iota_n$ is the canonical inclusion of $\fg$, which is a strict weak equivalence since $\Omega_n$ is acyclic. The morphism
		\[
		W_n^\sharp \otimes_\R \fg \longrightarrow \fg^{\oplus (n+1)}
		\]
		is given by the composition
		\[
		W_n^\sharp \otimes_\R \fg
		\xlongrightarrow{(\bfi_\mu)_n}
		\Omega_n^\sharp \,\widehat{\otimes}_{\R}\, \fg
		\xlongrightarrow{\ev_n}
		\fg^{\oplus (n+1)},
		\]
		where the second morphism $\ev_n$ is induced by the evaluation maps at the vertices of $|\Delta^n|$ on $\Omega_n^\sharp$, hence a strict fibration. Recall that the linear component of $(\bfi_\mu)_n$ is given by the canonical inclusion
		\[
		i_n \otimes \id:
		W_n^\sharp \otimes_\R \fg
		\longrightarrow
		\Omega_n^\sharp \,\widehat{\otimes}_{\R}\, \fg;
		\]
		see Proposition \ref{thht}. It follows that the linear component of the composite 
		$W_n^\sharp \otimes_\R \fg \rightarrow \fg^{\oplus (n+1)}$
		is precisely the map induced by the evaluation maps at the vertices of $|\Delta^n|$ on $W_n^\sharp$, and is therefore surjective.
		
		The diagram commutes since 
		\[
		(\bfi_\mu)_n \circ (\bfp_\mu)_n = \id -[D_n, (\bfs_\mu)_n]
		\]
		where $(\bfs_\mu)_n = (\id + \bfs_n \mu)^{-1}\bfs_n = \bfs_n (\id + \mu \bfs_n)^{-1}$, and
		\[
		\ev_n \circ [D_n, (\bfs_\mu)_n] \circ \iota_n  = 0,
		\]
		where we use $\bfs_n \circ \iota_n = \ev_n \circ \bfs_n = 0$. Indeed, $\bfs_n \circ \iota_n = 0$ because $h^i_n(1)=0$, while $\ev_n \circ \bfs_n = 0$ because $\varepsilon_{i}^n h_n^i = 0$ and $h_n^i h_n^j + h_n^j h_n^i = 0$.
	\end{proof}	
	
	\begin{prop}
		The category $\mathbf{L_\infty Alg}(R_\sM)_{\mathrm{fgp}}$ of finitely generated projective curved $L_\infty$ algebras over $R_\sM=(\wOmega_\sM(M), D_\sM)$ is a CFO with a framing.
	\end{prop}
	
	\begin{defn}
		A morphism $\phi=(\id, \phi^\sharp)\colon B\fg=(\sM, \fg) \rightarrow B\fh=(\sM, \fh)$ in $\mathbf{L_\infty Sp}(\sM)$ is a \emph{fibration} if $\phi^\sharp_M\colon \fg[1](M) \rightarrow \fh[1](M)$ is surjective.
	\end{defn}
	
	A functor between CFOs is \emph{exact} if it preserves fibrations, trivial fibrations, the terminal object, and pullbacks along fibrations. 
	
	\begin{thm}
		The category $\mathbf{L_\infty Sp}(\sM)$ of curved $L_\infty$ spaces over a dg manifold $\sM$ is a CFO with a framing. Moreover, the global sections functor
		\[
		\Gamma\colon \mathbf{L_\infty Sp}(\sM) \longrightarrow \mathbf{L_\infty Alg}(R_\sM)_{\mathrm{fgp}}
		\]
		is a fully faithful exact functor.
	\end{thm}
	\begin{proof}
		By Proposition \ref{gsf}, we may use the global sections functor $\Gamma$ to identify $\mathbf{L_\infty Sp}(\sM)$ with the full subcategory $\mathbf{L_\infty Alg}(R_\sM)_{\mathrm{vec}}$ of $\mathbf{L_\infty Alg}(R_\sM)_{\mathrm{fgp}}$. It is clear that the simplicial frame construction of Proposition \ref{simfr} preserves $\mathbf{L_\infty Alg}(R_\sM)_{\mathrm{vec}}$. By Proposition \ref{kerOmega}, the same is true for the pullback construction of Proposition \ref{pullback}.
	\end{proof}

	\appendix
	
	\section{Homological perturbation of curved complexes}\label{sec:hpt}
	
	We begin by fixing our notion of curved complexes and their morphisms.
	\begin{defn}
		A \emph{curved complex} is a filtered graded vector space $V$ equipped with a degree $1$ linear map
		\[
		d\colon V \longrightarrow V
		\]
		that is compatible with the filtration and satisfies
		\[
		d^2(x) \in F^{p+1}V \quad \text{for all } x \in F^pV.
		\]
		We call the linear map $r \coloneqq d^2$ the \emph{curvature} of $(V,d)$.
		
		A \emph{morphism} between curved complexes $(V_1, d_1)$ and $(V_2, d_2)$ is a degree $0$ linear map
		\[
		f\colon V_1 \longrightarrow V_2
		\]
		that is compatible with the filtration and satisfies
		\[
		(f d_1 - d_2f)(x)  \in F^{p+1}V_2 \quad \text{for all } x \in F^pV_1.
		\]
	\end{defn}
	
	Let $(V,d)$ be a curved complex with curvature $r$.
	
	\begin{defn}
		A \emph{gauge} of $(V,d)$ is a linear map of degree $-1$
		\[
		g \colon V \longrightarrow V
		\]
		compatible with the filtration, such that
		\[
		g^2 = [g,r] = 0
		\quad \text{and} \quad
		[d,g] \text{~is~semisimple}.
		\]
		A gauge $g$ is called \emph{idempotent} if $\operatorname{spec}([d,g])=\{0,1\}$.
	\end{defn}
	
	The \emph{Green's operator} $G \colon V \to V$ associated to a gauge $g$ is defined by
	\[
	G(v) = \begin{cases}
		0 & v \in V_0, \\
		\lambda^{-1} v & v \in V_\lambda, \quad \lambda \neq 0,
	\end{cases}
	\]
	where $V_\lambda \coloneqq \{x \in V \colon [d,g](x) = \lambda x\}$. It satisfies
	\[
	[d, G] = [g, G] = [r,G] = 0, 
	\]
	and
	\[
	[d,g]G(v) = \begin{cases}
		0 & v \in V_0 ,\\
		v & v \in V_\lambda, \quad  \lambda \neq 0.
	\end{cases}
	\]
	Consequently, the operator
	\[
	h(g)\coloneqq gG = Gg
	\]
	defines an idempotent gauge, which we call the \emph{harmonization} of $g$. 
	One has
	\[
	h(g) d h(g) = gGdgG = gG[d,g]G-gGgdG = gG - g^2 GdG = h(g),
	\]
	where we use $g^2= 0$ and $[g,G]=0$. 
	
	\begin{defn}
		A \emph{(homotopy) contraction} of $(V,d)$ is a map of degree $-1$ 
		\[
		h \colon V \longrightarrow V
		\]
		compatible with the filtration, such that
		\[
		h^2 = [h,r] = 0
		\quad \text{and} \quad
		hdh = h.
		\]
	\end{defn}
	
    A direct computation shows that
	\[
	[d,h]^2= [d,h].
	\]
	Therefore, every contraction defines an idempotent gauge, and every gauge gives rise to a contraction after harmonization.
	
	\begin{defn}
		A \emph{strong deformation retract} of curved complexes from $(V,d)$ to $(H, \delta)$ consists of morphisms of curved complexes
		\[
		p \colon (V,d) \longrightarrow (H, \delta),
		\qquad
		i \colon (H,\delta) \longrightarrow (V,d),
		\]
		and a map of degree $-1$
		\[
		h \colon V \longrightarrow V
		\]
		compatible with the filtration, such that
		\begin{equation*}
			\id - ip = [d,h],
			\qquad
			\id - pi = 0,
			\qquad
			h^2 = ph = hi = 0.
		\end{equation*}
	\end{defn}
	
	Every strong deformation retract is determined by a contraction up to isomorphism. Indeed, given a contraction $h$, the idempotent operator $[d,h]$ defines a splitting of $V$ into
	\[
	V = \ker([d,h]) \oplus \operatorname{im}([d,h]).
	\]
	Setting $H = \ker([d,h])$, we obtain canonical morphisms of curved complexes
	\[
	i \colon (H, \delta) \longrightarrow (V,d), \qquad p \colon (V,d) \longrightarrow (H, \delta),
	\]
	given respectively by the inclusion and projection, where $\delta = pdi$. 
	
	By definition, $i$ and $p$ satisfy
	\[
	\id - ip = [d,h],\quad
	\id - pi = 0.
	\]
	Composing with $h$ yields
	\[
	h = hip + hdh = iph + hdh.
	\]
	Since $hdh = h$ and $pi = \id$, this is equivalent to the side conditions
	\[
	ph = hi = 0.
	\]
	\begin{rmk}
		It is possible to drop the axiom \([h,r]=0\) from the definition of a contraction. In this case, the operator
		\[
		[d,h]-hrh
		\]
		is idempotent. Indeed, using the identity \(hdh=h\), one computes
		\[
		[d,h]^2 = [d,h] + hrh,
		\]
		and
		\[
		([d,h]-hrh)^2
		= [d,h] + hrh - 2hrh
		= [d,h]-hrh.
		\]
		
		Correspondingly, the first identity in the definition of a strong deformation retract must be modified to
		\[
		\id - ip = [d,h] - hrh.
		\]
		This version of a strong deformation retract is considered in \cite{amorim2022inverse} to study homotopy transfers of curved $L_\infty$ spaces over manifolds.
	\end{rmk}

	\begin{defn}
		A \emph{perturbation} of $(V,d)$ is a map of degree $1$
		\[
		\mu \colon F^\bullet V \longrightarrow F^{\bullet + 1} V
		\]
		such that
		\[
		[d, r_\mu]=0,
		\]
		where $r_\mu\coloneqq(d+\mu)^2$.
	\end{defn}
	
	Since the filtration on $V$ is complete, the operators $\id+h\mu$ and $\id+\mu h$ are invertible, with inverses given by the convergent formal series
	\[
	(\id+h\mu)^{-1}
	=
	\sum_{n=0}^{\infty}(-h\mu)^n,
	\qquad
	%=
	%\id-h\mu+h\mu h\mu-\cdots,
	%\]
	%and
	%\[
	(\id+\mu h)^{-1}
	=
	\sum_{n=0}^{\infty}(-\mu h)^n.
	%=
	%\id-\mu h+\mu h\mu h-\cdots.
	\]
	In particular, one has
	\[
	(\id+\mu h)^{-1}\mu
	=
	\mu(\id+h\mu)^{-1},
	\quad
	h(\id+\mu h)^{-1}
	=
	(\id+h\mu)^{-1}h,
	\]
	and
	\[
	(\id+\mu h)^{-1}(\id+ h \mu)^{-1} = (\id+\mu h)^{-1} + (\id+ h \mu)^{-1} - \id.
	\]
	\begin{lem}\label{sdrh}
		Let $h$ be a contraction of $(V,d)$, and let $\mu$ be a perturbation of $(V,d)$ satisfying
		\[
		[h,r_\mu]=0.
		\]
		Set
		\[
		d_\mu \coloneqq d+\mu,
		\qquad
		h_\mu \coloneqq (\id+h\mu)^{-1}h.
		\]
		Then $h_\mu$ is a contraction of the perturbed curved complex $(V,d_\mu)$.
	\end{lem}
	\begin{proof}
		We have $h_\mu^2 = (\id + h\mu)^{-1} h^2 (\id + \mu h)^{-1} = 0$, and
		\begin{align*}
			h_\mu d_\mu h_\mu &= (\id + h\mu)^{-1} h(d+\mu) (\id + h\mu)^{-1} h \\
			&= (\id+ h\mu)^{-1}h (d+\mu) h(\id + \mu h)^{-1} \\
			&= (\id+ h\mu)^{-1}(h+h\mu h)(\id + \mu h)^{-1} \\
			&= (\id+ h\mu)^{-2}(h+h\mu h) \\
			%&=(\id+ h\mu)^{-1}h \\
			&= h_\mu,
		\end{align*}
		Finally, note that
		\[
		[\mu, r_\mu] = [d_\mu, r_\mu] = [d_\mu, d_\mu^2]=0.
		\]
		It then follows from $[h, r_\mu]=0$ that $[h_\mu, r_\mu]=0$.
	\end{proof}

	%\[
	%\delta_\mu \coloneqq \delta + p(\id + \mu h)^{-1}\mu i = \delta + \id - p \mu h \mu i + p \mu h \mu h \mu i - \cdots.
	%\]
	
	\begin{prop}%[Homological perturbation lemma]
		Let $h$ and $\mu$ be as in Lemma \ref{sdrh}, and set
		\begin{align*}
			i_\mu \coloneqq (\id + h\mu)^{-1} i, 
			\qquad
			p_\mu \coloneqq p(\id + \mu h)^{-1},
			\qquad 
			\delta_\mu \coloneqq \delta + p_\mu \mu i.
		\end{align*}
		Then the triple $(p_\mu,i_\mu,h_\mu)$ defines a strong deformation retract of curved complexes from $(V,d_\mu)$ to $(H,\delta_\mu)$.
	\end{prop}
	
	\begin{proof}
		It suffices to show that $p_\mu$ and $i_\mu$ are the correct projection and inclusion and $\delta_\mu$ is the correct differential induced by them. We have
		\begin{align*}
			p_\mu i_\mu &= p(\id+\mu h)^{-1}(\id+ h \mu)^{-1} i \\
			& =p (\id+\mu h)^{-1}i + p(\id+ h \mu)^{-1}i  - pi\\
			&= \id,
		\end{align*}
		where we use $pi = \id$ and the side conditions $hi=ph=0$. On the other hand,
		\begin{align*}
			(\id+ h \mu)i_\mu p_\mu (\id+\mu h)= ip =\id- (dh+hd),
		\end{align*}
		and
		\begin{align*}
			(\id+ h \mu)(d_\mu h_\mu + h_\mu d_\mu) (\id+\mu h) &= (\id+ h \mu) d_\mu h + h d_\mu (\id+\mu h) \\
			&= (\id+ h \mu) (\id+\mu h) + dh+hd - \id,
		\end{align*}
		where we use $d \mu + \mu d + \mu^2 = r_\mu -r$ and $h (r_\mu-r) h = [h,r_\mu-r]h=0$. Since $\id+\mu h$ and $\id + h\mu$ are invertible, it follows that
		\[
		i_\mu p_\mu + d_\mu h_\mu + h_\mu d_\mu = \id.
		\]
		
		Finally, we compute that
		\[
		(\id+\mu h)d(\id+h\mu) = d + \mu h d + d h \mu + \mu h \mu,
		\]
		and
		\begin{align*}
			&(\id+\mu h)^{-1}d_\mu(\id+h\mu)^{-1} = d - (\id+\mu h)^{-1}(\mu h d + d h \mu + \mu h \mu -\mu)(\id+h\mu)^{-1} \\
			&= d - (\id+\mu h)^{-1}(\mu h d + d h \mu + 2 \mu h \mu)(\id+h\mu)^{-1} + (\id+\mu h)^{-1}\mu.
		\end{align*}
		Thus, 
		\[
		\delta_\mu  - p_\mu d_\mu i_\mu =p_\mu (\mu h d + d h \mu + 2 \mu h \mu)i_\mu.
		\]
		
		Note that
		\begin{align*}
			\mu h d + d h \mu &= \mu [d,h] + [d, h]\mu - \mu d h - h d \mu \\
			&=2 \mu - \mu ip - ip \mu -\mu d h - h d \mu,
		\end{align*}
		where we use $\id-ip =[d,h]$, and
		\[
		p_\mu h= 0, \qquad h i_\mu= 0, \qquad p_\mu i = p i_\mu  =  \id,
		\]
		where we use the side conditions $h^2 = ph = hi =0$. Therefore,
		\[
		p_\mu (\mu h d + d h \mu + 2 \mu h \mu )i_\mu = 2 p_\mu \mu i_\mu - p \mu i_\mu - p_\mu \mu i + 2 p_\mu \mu h \mu i_\mu = 0,
		\]
		where we use $p_\mu = p - p_\mu \mu h$ and $i_\mu = i - h \mu i_\mu$.
	\end{proof}
	
	In particular, one may consider perturbations $\mu$ such that $r_\mu=0$, that is, such that $d_\mu$ (and hence $\delta_\mu$) is a differential. In this case, the conditions $[d,r_\mu]=[h,r_\mu]=0$ on $\mu$ are trivially satisfied.

	\section{Homotopy transfer of curved \texorpdfstring{$L_\infty$}{L-infinity} algebras}\label{sec:htt}
	
	Let $\fg$ be a curved $L_\infty$ algebra over $R=(\uR, D_R)$. 
	%Then $(\fg, l_1)$ is a curved complex with $R$-linear curvature
	%\[
	%r\coloneqq-l_2(l_0, \cdot).
	%\]
	\begin{defn}
		A \emph{contraction} of $\fg$ is an $\uR$-linear contraction $h$ of the curved complex $(\fg[1], l_1)$ such that
		\[
		h(l_0)=0.
		\]
	\end{defn}
	
	Let $h$ be a contraction of $\fg$. Let $L_h$ denote the degree $-1$ $\uR$-linear coderivation on $\wSym_\uR(\fg[1])$ extending $h\colon \fg[1] \to \fg[1]$. Explicitly, $L_h$ is given by
	\[
	L_h(x_1 \odot \cdots \odot x_n)
	=
	\sum_{k=1}^n (-1)^{\sum_{i=1}^{k-1}|x_i|}\,
	x_1 \odot \cdots \odot h(x_k)\odot \cdots \odot x_n.
	\]
	Consider the operator $\bfd$ on $\wSym_\uR(\fg[1])$ defined by
	\begin{align*}
		&\bfd(r)= D_R(r) + (-1)^{|r|}rl_0, \\
		&\bfd (x_1 \odot \cdots \odot x_n)
		= l_0 \odot x_1 \odot \cdots \odot x_n + \sum_{k=1}^n (-1)^{\sum_{i=1}^{k-1}|x_i|}\,
		x_1 \odot \cdots \odot l_1(x_k)\odot \cdots \odot x_n.
	\end{align*}
	Equipped with $\bfd$, $\wSym_\uR(\fg[1])$ becomes a curved complex with $\uR$-linear curvature
	\[
	\bfr \coloneqq L_{r},
	\]
	where $r \coloneqq -l_2(l_0, \cdot)$ is the $\uR$-linear curvature of $(\fg,l_1)$.
	
	\begin{lem}
		$L_h$ is an $\uR$-linear gauge of the curved complex $(\wSym_\uR(\fg[1]), \bfd)$.
	\end{lem}
	\begin{proof}
		We have $L_h^2 = \frac{1}{2}[L_h,L_h] = \frac{1}{2}L_{[h,h]} = 0$, and $[L_h, \bfr]=L_{[h,r]}=0$. The semi-simplicity of $[\bfd, L_h]$ follows from the semi-simplicity of $[l_1,h]$ and the computation
	\begin{equation}\label{ll1h}
		[\bfd, L_h] =  L_{[l_1,h]}, 
	\end{equation}
	where we use the assumption $h(l_0)=0$.
	\end{proof}
	
	Let $\fh[1] \coloneqq \ker [l_1,h]$. Let $p\colon \fg[1] \rightarrow \fh[1]$ and $i\colon \fh[1] \rightarrow \fg[1]$ denote the projection and inclusion, respectively. Let 
	\[
	\bfp\colon  \wSym_\uR(\fg[1]) \rightarrow  \wSym_\uR(\fh[1]), \qquad \bfi\colon \wSym_\uR(\fh[1]) \rightarrow   \wSym_\uR(\fg[1])
	\]
	denote the projection and inclusion of filtered graded cocommutative coalgebras over $\uR$ induced by $p$ and $i$. 
	
	Let $G$ denote Green's operator of $[\bfd, L_h]$. 
	Berglund derived an explicit formula for $G$, which he formulated in terms of symmetric thick maps \cite{berglund2014homological}.
	\begin{lem}\label{berg}
		For $n>0$,
		\[
		G\big|_{\Sym_\uR^n(\fg[1])} =\frac{1}{n} \sum_{\epsilon \in \{0,1\}^n} \binom{n-1}{|\epsilon|}^{-1} (ip)^{\epsilon_1} \otimes \cdots \otimes (ip)^{\epsilon_n},
		\]
		with the convention $\binom{n-1}{n}^{-1}=0$.
	\end{lem}
	\begin{proof}
	We need to show that
	\[
	[\bfd,L_h]G\big|_{\Sym_\uR^n(\fg[1])} =\id^{\otimes n}-(ip)^{\otimes n}.
	\]
	Using \eqref{ll1h} and the identity $[l_1,h]=\id-ip$, this is equivalent to proving
	\[
	\left(\sum_{i+1+j=n}
	\id^{\otimes i}\otimes(\id-ip)\otimes\id^{\otimes j}\right)
	\left(
	\frac{1}{n}
	\sum_{\epsilon\in\{0,1\}^n}
	\binom{n-1}{|\epsilon|}^{-1}
	(ip)^{\epsilon_1}\otimes\cdots\otimes(ip)^{\epsilon_n}
	\right)
	=
	\id^{\otimes n}-(ip)^{\otimes n}.
	\]
	
	The left-hand side can be simplified as 
	\begin{align*}
		\mathrm{LHS}
		&=
		\frac1n
		\sum_{j=1}^n
		\sum_{\substack{\epsilon\in\{0,1\}^n \\ \epsilon_j=0}}
		\binom{n-1}{|\epsilon|}^{-1}
		(ip)^{\epsilon_1}\otimes\cdots\otimes(\id-ip)\otimes\cdots\otimes(ip)^{\epsilon_n}\\
		&=
		\frac1n
		\sum_{j=1}^n
		\sum_{\epsilon\in\{0,1\}^n}
		(-1)^{\epsilon_j}
		\binom{n-1}{|\epsilon|-\epsilon_j}^{-1}
		(ip)^{\epsilon_1}\otimes\cdots\otimes(ip)^{\epsilon_n}\\
		&=
		\sum_{\epsilon\in\{0,1\}^n}
		c_\epsilon\,
		(ip)^{\epsilon_1}\otimes\cdots\otimes(ip)^{\epsilon_n},
	\end{align*}
	where
	\[
	c_\epsilon
	\coloneqq
	\frac1n
	\sum_{j=1}^n
	(-1)^{\epsilon_j}
	\binom{n-1}{|\epsilon|-\epsilon_j}^{-1}.
	\]
	Since there are $n-|\epsilon|$ indices with $\epsilon_j=0$ and $|\epsilon|$ indices with $\epsilon_j=1$,
	\[
	c_\epsilon
	=
	\frac1n
	\left(
	\frac{n-|\epsilon|}{\binom{n-1}{|\epsilon|}}
	-
	\frac{|\epsilon|}{\binom{n-1}{|\epsilon|-1}}
	\right).
	\]
	Thus $c_\epsilon$ depends only on $|\epsilon|$; we denote it by $c_{|\epsilon|}$.
	
	For \(0<|\epsilon|<n\),
	\[
	\frac{n-|\epsilon|}{\binom{n-1}{|\epsilon|}}
	=
	\frac{|\epsilon|!(n-|\epsilon|)!}{(n-1)!}
	=
	\frac{|\epsilon|}{\binom{n-1}{|\epsilon|-1}},
	\]
	and hence $c_{|\epsilon|}=0$. For the remaining cases, we have
	\[
	c_0=\binom{n-1}{0}^{-1}=1,
	\qquad
	c_n=-\binom{n-1}{n-1}^{-1}=-1,
	\]
	which completes the proof.
	\end{proof}
	
	Let $\bfh$ be the harmonization of $L_h$. That is,
	\begin{equation}\label{bfh}
		\bfh = L_h G = G L_h.
	\end{equation}
	Comparing with Proposition 5.1 of \cite{berglund2014homological}, one sees immediately that \eqref{bfh} coincides with the symmetrized tensor trick homotopy constructed by Berglund.\footnote{In his notation, $\bfh^\Sigma$ corresponds to our $\bfh$, $\bfq$ corresponds to $G$, and $\bfh^{\mathrm{der}}$ corresponds to $L_h$.}

	Now consider the operator $\bfdelta$ on $\wSym_\uR(\fh[1])$ defined by
	\begin{align*}
		&\bfdelta(r)= D_R(r) + (-1)^{|r|}r p(l_0), \\
		&\bfdelta (x_1 \odot \cdots \odot x_n)
		= p(l_0) \odot x_1 \odot \cdots \odot x_n + \sum_{k=1}^n (-1)^{\sum_{i=1}^{k-1}|x_i|}\,
		x_1 \odot \cdots \odot pl_1i(x_k)\odot \cdots \odot x_n.
	\end{align*}
	Likewise, equipped with $\bfdelta$, $\wSym_\uR(\fh[1])$ becomes a curved complex with curvature $\bfp \bfr \bfi$. Indeed, this follows from the computation
	\[
	pl_1 ip l_1 i = p l_1(\id - [l_1, h])l_1 i = pri - p(rhl_1+l_1hr)i = pri - pr(\id-ip)i = pri,
	\]
	where we use $[l_1, r]=[h,r]=0$.
	
	\begin{lem}
		The triple $(\bfp, \bfi, \bfh)$ defines a strong deformation retract of curved complexes from $(\wSym_\uR(\fg[1]), \bfd)$ to $(\wSym_\uR(\fh[1]), \bfdelta)$.
	\end{lem}
	\begin{proof}
		It suffices to verify that $\bfp$ and $\bfi$ are the correct projection and inclusion induced by $\bfh$, and that $\bfdelta$ is the differential induced by the data. The identities $\id - pi =0$ and $\delta = p d i$ immediately imply that $\id - \bfp \bfi = 0$ and $\bfdelta = \bfp \bfd \bfi$. The identity $\id - \bfi\bfp = [\bfd, \bfh]$ follows from Lemma \ref{berg}. Indeed,
		\[
		[\bfd, \bfh] = [\bfd, L_h] G = \id - \bfi\bfp,
		\]
		where we use $[\bfd, G]=0$ and $[\bfh, G]=0$.
	\end{proof}
	
	Now define
	\[
	\mu \coloneqq D_\fg - \bfd.
	\]
	$\mu$ is explicitly given by
	\begin{align*}
		\mu(x_1 \odot \cdots \odot x_n)=
		\sum_{\sigma \in \Sigma_n} \sum_{k=2}^n 
		\frac{\epsilon(\sigma)}{k!(n-k)!}
		l_k(x_{\sigma(1)},\dots,x_{\sigma(k)})
		\odot x_{\sigma(k+1)} \odot \cdots \odot x_{\sigma(n)}.
  	\end{align*}
  	Note that $\mu$ increases the filtration degree. Thus, it defines an $\uR$-linear perturbation of $\bfd$. 
	%The perturbed differential is the Chevalley--Eilenberg differential of $\fg$.
	The corresponding perturbed projection, inclusion, and contraction are given by
	\begin{align*}
		\bfp_\mu \coloneqq \bfp\,(\id + \mu \bfh)^{-1}, \quad
		\bfi_\mu \coloneqq (\id + \bfh \mu)^{-1}\bfi, \quad
		\bfh_\mu \coloneqq (\id + \bfh\mu)^{-1}\bfh.
	\end{align*}
	%\begin{rmk}
		%$\Gr\, \bfp_\mu =\Gr\, \bfp$
	%\end{rmk}
	By homological perturbation theory,	the triple $(\bfp_\mu, \bfi_\mu, \bfh_\mu)$ defines a strong deformation retract of curved complexes from $(\wSym_\uR(\fg[1]), D_\fg)$ to $(\wSym_\uR(\fh[1]), D_\fh)$, where
	\[
	D_\fh \coloneqq \bfdelta + \bfp_\mu \mu \bfi.
	\]
	
	\begin{lem}\label{bergetz}
		The filtration-preserving $\uR$-linear maps
		\[		
		\bfp_\mu\colon\wSym_\uR(\fg[1]) \longrightarrow \wSym_\uR(\fh[1])
		\qquad
		\bfi_\mu\colon\wSym_\uR(\fh[1]) \longrightarrow \wSym_\uR(\fg[1])
		\]
		are morphisms of filtered graded cocommutative coalgebras over $\uR$. Moreover, the differential $D_\fh$ is a coderivation of $\wSym_\uR(\fh[1])$ compatible with $D_R$.
	\end{lem}
	The proof follows \cite{berglund2014homological}; see also Theorem 4.1 in \cite{getzler2025higherholonomycurvedlinftyalgebras}.
	\begin{proof}		
		Using the side conditions $h^2=ph = hi=0$, the fact that $L_h$ is a coderivation, and  \eqref{bfh}, one easily obtains
		\[
		(\bfh \otimes \bfh)\,\Delta \bfh = (\bfp \otimes \bfh)\,\Delta \bfh = (\bfh \otimes \bfp)\,\Delta \bfh = (\bfp \otimes \bfp)\,\Delta \bfh = 0.
		\]
		It follows that
		\begin{align*}
			(\bfp_\mu \otimes \bfp_\mu)\,\Delta \bfh = ((\bfp - \bfp_\mu \mu \bfh) \otimes (\bfp - \bfp_\mu \mu \bfh))\,\Delta \bfh = 0,
		\end{align*}
		and
		\[
		 (\bfp_\mu \otimes \bfp_\mu)\,\Delta [D_\fg, \bfh]=(\bfp_\mu \otimes \bfp_\mu)(\id \otimes D_\fg + D_\fg \otimes \id) \Delta  \bfh = (\id \otimes D_\fh + D_\fh \otimes \id) (\bfp_\mu \otimes \bfp_\mu) \Delta  \bfh = 0.
		\]
		On the other hand,
		\[
		[D_\fg, \bfh] = [\bfd + \mu, \bfh] = \id - \bfi \bfp + [\mu, \bfh].
		\]
		We then have
		\begin{align*}
			(\bfp_\mu \otimes \bfp_\mu) \Delta &= (\bfp_\mu \otimes \bfp_\mu) \Delta \left([D_\fg, \bfh]  + \bfi \bfp - [\mu, \bfh]\right) \\
			&= (\bfp_\mu \bfi \otimes \bfp_\mu \bfi) \Delta \bfp  - (\bfp_\mu \otimes \bfp_\mu) \Delta \mu \bfh \\
			&=\Delta \bfp  - (\bfp_\mu \otimes \bfp_\mu) \Delta \mu \bfh,
		\end{align*}
		where we use $\bfp_\mu \bfi =  \bfp \bfi =\id$. Since $\id + \mu \bfh$ is invertible, we obtain
		\[
		(\bfp_\mu \otimes \bfp_\mu)\,\Delta = \Delta \bfp_\mu.
		\]
		
	    Using Lemma \ref{berg} and the side conditions $ph = hi = 0$,  one has
		\begin{align*}
			\bfh \mu \bfi \bfp (x_1 \odot \cdots \odot x_n) &= \sum_{\sigma \in \Sigma_n} \sum_{k=2}^n \frac{\epsilon(\sigma)}{k!(n-k)!}
			\frac{1}{n-k+1} \sum_{\substack{\epsilon \in \{0,1\}^{n-k+1}\\ \epsilon_1 =0}} \binom{n-k}{|\epsilon|}^{-1} \\
			&
			hl_k(ip(x_{\sigma(1)}),\dots,ip(x_{\sigma(k)}))
			\odot ip(x_{\sigma(k+1)}) \odot \cdots \odot ip(x_{\sigma(n)}).
		\end{align*}
		Note that
		\[
		\frac{1}{m}
		\sum_{\substack{\epsilon\in\{0,1\}^{m}\\ \epsilon_1=0}}
		\binom{m-1}{|\epsilon|}^{-1}
		=
		\frac{1}{m}
		\sum_{r=0}^{m-1}
		\binom{m-1}{r}\binom{m-1}{r}^{-1}
		=
		\frac{1}{m}\sum_{r=0}^{m-1}1
		=1.
		\]
		Hence, $\bfh \mu \bfi \bfp  = L_h \mu \bfi \bfp$.
		
		For $k > 1$, a similar computation yields 
		\[
		(\bfh \mu)^k \bfi \bfp = \frac{1}{k!} (L_h \mu)^k\bfi \bfp = \frac{1}{k!}  [L_h, \mu]^k\bfi \bfp.
		\]
		Since $[L_h, \mu]$ is a coderivation, we have
		\[
		\Delta (\bfh \mu)^k \bfi \bfp = \frac{1}{k!} \sum_{p=0}^k \binom{k}{p}\left((L_h \mu)^p \otimes (L_h \mu)^{k-p} \right) \Delta \bfi \bfp = \sum_{p+q=k}\left((\bfh \mu)^p \otimes (\bfh\mu)^q \right) \Delta \bfi \bfp.
		\]
		Therefore,
		\begin{align*}
			\Delta\bfi_\mu
			&=\Delta(\id+\bfh\mu)^{-1} \bfi \bfp \bfi\\
			&= \left((\id+\bfh\mu)^{-1}\bfi \otimes (\id+\bfh\mu)^{-1}\bfi \right) \Delta  \bfp \bfi \\
			&= \left(\bfi_\mu \otimes \bfi_\mu \right) \Delta.
		\end{align*}
		
		$D_\fh = \bfdelta + \bfp_\mu \mu \bfi$ is a coderivation since both $\bfdelta$ and $\mu$ are coderivations, while $\bfp_{\mu}$ and $\bfi$ are coalgebra morphisms. It is compatible with $D_R$ because $\bfdelta$ is compatible with $D_R$, and the term $\bfp_{\mu} \mu \bfi$ is $\uR$-linear.
	\end{proof}
	
	\begin{prop}\label{thht}
		A contraction $h$ of $\fg$ induces a curved $L_\infty$ algebra structure on $\fh \subset \fg$ with curvature and unary bracket given by
		\[
		p(l_0), \qquad p l_1i,
		\]
		respectively, and morphisms of curved $L_\infty$ algebras
		\[
		\bfp_\mu\colon \fg \to \fh,
		\qquad
		\bfi_\mu\colon \fh \to \fg,
		\]
		satisfying
		\[
		(\bfp_\mu)_0 = 0,
		\qquad
		(\bfi_\mu)_0 = 0,
		\qquad
		(\bfp_\mu)_1 = p,
		\qquad
		(\bfi_\mu)_1 = i.
		\]
		In particular, $\bfp_\mu$ and $\bfi_\mu$ are weak equivalences of curved $L_\infty$ algebras.
	\end{prop}
	\begin{proof}
		By Lemma~\ref{bergetz}, $\fh$ carries a curved $L_\infty$ algebra structure with Chevalley--Eilenberg differential $D_\fh$, and $\bfp_\mu$ and $\bfi_\mu$ are well defined morphisms of curved $L_\infty$ algebras. The zeroth and first components of $\bfp_\mu$, $\bfi_\mu$, and $D_\fh$ can be read off from their action on $\fg[1]\subset \widehat{\Sym}_\uR(\fg[1])$ and $\fh[1]\subset \widehat{\Sym}_\uR(\fh[1])$, respectively.
		
		$\bfp_\mu$ and $\bfi_\mu$ are weak equivalences since the triple $(\Gr\,p, \Gr\,i, \Gr\,h)$ defines a strong deformation retract of complexes from $(\Gr\,\fg,\Gr\,l_1)$ to $(\Gr\,\fh,\Gr\,pl_1i)$.
	\end{proof}
	
	We end this appendix by emphasizing that our homotopy transfer of curved $L_\infty$ algebras keeps the curvature unchanged. Indeed, one has
	\[
	pi(l_0) = l_0, \qquad
	ip(l_0)= l_0 - [l_1,h](l_0)
	= l_0 - h(l_1(l_0))
	= l_0.
	\]
	%In other words, the curvature of $\fg$ must lie in the subspace $\fh \subset \fg$.

	%2602.24099

\end{document}